%% file: char_coh_14.07.17.tex
\documentclass[11pt]{amsart}
\usepackage{amssymb,latexsym,amsmath,longtable,mathdots}
\usepackage[mathscr]{eucal}
\usepackage{bbm}
\usepackage{color,lscape}

\numberwithin{equation}{section}

\input{macros.tex}
\input{thms.tex}

\def\bd{\mathbf{d}}
\def\tss{\mathrm{ss}}
\def\Wvhs{W_\mathrm{vhs}}

\def\FHW{MR2188135}
\def\GGK{MR2918237}
\def\Kostant1{MR0142696}

\begin{document}
\title[Characteristic cohomology of the IPR]{Characteristic Cohomology of the Infinitesimal Period Relation}
\author[Robles]{C. Robles}
\email{robles@math.tamu.edu}
\address{Mathematics Department, Mail-stop 3368, Texas A\&M University, College Station, TX  77843-3368} 

\date{\today}

\begin{abstract}
The infinitesimal period relation (also known as Griffiths' transversality) is the system of partial differential equations constraining variations of Hodge structure.  This paper presents a study of the characteristic cohomology associated with that system of PDE.
\end{abstract}
\maketitle

\setcounter{tocdepth}{1}

\section{Introduction}

Let $\check D = G_\bC/P$ be a (generalized) flag variety; here $G_\bC$ is a complex, semisimple Lie group and $P$ is a parabolic subgroup.\footnote{The notation $\check D$ for $G_\bC/P$ comes from Hodge theory: we think of $\check D$ as the compact dual of a period domain (or, more generally, a Mumford--Tate domain).}  The topic of this paper is the characteristic cohomology associated with a differential system on $\check D$.  The differential system is given by the unique minimal $G_\bC$--homogeneous bracket--generating subbundle $\cT_1 \subset \cT\check D$ of the holomorphic tangent bundle.  The equality $\cT_1 = \cT\check D$ holds if and only if $\check D$ admits the structure of a compact Hermitian symmetric space.  In all other cases, bracket--generation implies the distribution is as far from integrable (or Frobenius) as it is possible to be.  

A connected complex submanifold $M \subset \check D$ is a solution if $\cT_xM \subset \cT_{1,x}$ for all $x\in M$.  Likewise, we will say that an irreducible variety $Y \subset \check D$ is a solution if $\cT_yY \subset \cT_{1,y}$ for all smooth points $y \in Y$.  Here, the case that $Y$ is a Schubert variety will be of particular interest.

Associated to this system is a differential ideal $\cI \subset \cA$ in the ring of differential forms with the property that $M$ is a solution if and only if $\left.\cI\right|_M = 0$.  Given any open subset $U \subset \check D$, the de Rham complex $(\cA_U,\td)$ induces a quotient complex, $(\cA_U/\cI_U,\td)$, and the characteristic cohomology $H^\sb_\cI(U) = H^\sb(\cA_U/\cI_U,\td)$ is the cohomology of this complex.  We may think of the characteristic cohomology as the cohomology that induces ordinary cohomology on integral manifolds $M\subset U$ by virtue of their being solutions of the system of differential equations.  

As will be discussed below, the characteristic cohomology may be realized as the  cohomology of a complex of differential operators.  The cohomology of a differential complex and related systems of differential equations is a subject of considerable interest (addressing such questions as: When is the cohomology finite dimensional? When does it vanish?  When does a local Poincar\'e Lemma hold?); see, for example, \cite{MR634855, MR0150342, BH2014, MR1311820, MR1334205, DanielMa, MR2217688, MR753407, MR662464}.  It should also be noted that the characteristic cohomology considered here is closely related to the \emph{characteristic cohomology of an exterior differential system} ($\mathrm{CC}_\mathrm{eds}$); indeed, we will be working with the ``Provisional Definition'' of R.~Bryant and P.~Griffiths's foundational \cite{MR1311820}.\footnote{The inadequacy of the provisional definition from the perspective of exterior differential systems is due to the necessity of considering derivatives all orders (notably for the purpose of identifying conservation laws).  For additional works on $\mathrm{CC}_\mathrm{eds}$ the reader is encouraged to consult \cite{MR1334205, MR2861607, MR1749440}.}

\subsection*{Characteristic cohomology on $\check D$}

The first set of results address the case that $U = \check D$.  We begin with the observation that the characteristic cohomology is spanned by the de Rham cohomology classes that are Poincar\'e--dual to the Schubert solutions (Theorem \ref{T:CCcD}).  Next we show that that a homology class on $\check D$ may be be represented by a union of solutions if and only if it may be represented by a union of Schubert solutions (Theorem \ref{T:hom}).  As a corollary we obtain a non-degenerate Poincar\'e--type pairing between the characteristic cohomology and the $\cI$--homology (Corollary \ref{C:deRham}).
 
\subsection*{Characteristic cohomology on flag domains $D \subset \check D$}

Motivated by Hodge theory, we next turn to the case that $D \subset \check D$ is a (generalized) flag domain; that is, $D$ is an open orbit of a real form $G_\bR$ of $G_\bC$.  When the isotropy group $G_\bR \cap P$ is compact, the group $G_\bR$ admits the structure of a Mumford--Tate group and flag domain may be realized as Mumford--Tate domain.  Mumford--Tate groups are the symmetry groups of Hodge theory: they arise as stabilizers of the Hodge tensors for a given Hodge structure.  Mumford--Tate domains generalize period domains and are the classification spaces for Hodge structures with (possibly) additional symmetry; see \cite{MR2918237} for details.   When restricted to a flag domain $D$, the subbundle $\cT_1$ is the infinitesimal period relation (also known as Griffiths' transversality), the differential constraint governing variations of Hodge structure.\footnote{In general the IPR will not be bracket--generating; however, one may always reduce to this case \cite[Section 3.3]{ MR3217458}.}  Suppose that $X \subset \Gamma \backslash D$ is (the image of) a variation of Hodge structure; here $\Gamma \subset G_\bR$ is a discrete subgroup acting properly discontinuously on $D$ so that the quotient $\Gamma \backslash D$ is a complex analytic variety, $X$ is K\"ahler and algebraic, and the local lifts of $X$ to $D$ are integrals of $\cT_1$.  The expectation is that Hodge structures on $X$ should arise universally; that is, should be induced from objects on $\Gamma \backslash D$.  In particular, it is anticipated that the characteristic cohomology induces a mixed Hodge structure on $X$.  (This is why we take what Bryant and Griffiths term the ``Provisional Definition'' of characteristic cohomology in \cite{MR1311820}.)  For more on distribution $\cT^1$ and the characteristic cohomology $H_\cI^\sb(D)$ from the perspective of Hodge theory see J.~Carlson, M.~Green and P.~Griffiths's recent \cite{MR2559674} and the references therein.  The invariant characteristic cohomology $H^\sb_\cI(D)^{G_\bR}$ is studied in \cite{MR3217458}; loosely speaking, this cohomology describes the topological invariants of global variations of Hodge structure that can be defined independently of the monodromy.

The main result of the paper for the characteristic cohomology on $D$ is the identification of an integer $\nu > 0$ with the property that $H^k_\cI(U) \simeq H^k(U)$ for all open $U \subset D$ and $k < \nu$ (Theorem \ref{T:hh} and \eqref{E:cvc}).  Corollary to the result we find that (i) the characteristic cohomology $H^k_\cI(D)$ is finite dimensional for all $k < \nu$ (Corollary \ref{C:FD}), and (ii) a local Poincar\'e lemma holds for differential of the characteristic cohomology in degree $k < \nu$ (Corollary \ref{C:PL3}).  The integer $\nu$ is given by Kostant's theorem on Lie algebra cohomology.  (A number of examples are discussed in Appendix \ref{S:egs}.)  The proof of Theorem \ref{T:hh} makes use of a realization of the characteristic cohomology on $D$ as the total cohomology of a double complex of $G_\bR$--invariant differential operators (Theorem \ref{T:cc}).  The fact that the characteristic cohomology can be realized as the cohomology of a complex of differential operators is not new; see, for example, J.~Daniel and X.~Ma's \cite{DanielMa}.  What is new in Theorem \ref{T:cc}, and is essential for the arguments establishing Theorem \ref{T:hh}, is the explicit representation theoretic description of the $G_\bR$--homogeneous bundles and $G_\bR$--invariant differential operators forming the complex.

\subsection*{Acknowledgements}

Over the course of this work I benefited from conversations and/or correspondence with a number of people including Andreas \Cap, Jeremy Daniel, Michael Eastwood, Phillip Griffiths, Mark Green, J.M. Landsberg, Carlos Simpson 
and \Vladimir~\Soucek; I thank them for their time and insight.  

I gratefully acknowledge partial support through NSF grants DMS-1006353, 1309238.  This work was completed while I was a member of the Institute for Advanced Study: I thank the institute for a wonderful working environment and the Robert and Luisa Fernholz Foundation for financial support.

\tableofcontents

\section{Flag varieties and flag domains}  \label{S:FD}

\noindent This section is a terse review of well--established material, serving primarily to introduce notation and conventions.  For more detail see \cite{\FHW, \GGK}.

A \emph{flag variety} (or \emph{flag manifold}) is a complex homogeneous space 
\[
  \check D \ = \ G_\bC/P
\]
where $G_\bC$ is a connected, complex semisimple Lie group and $P$ is a parabolic subgroup.  A familiar example is the Grassmannian $\tGr(k,\bC^n)$ of $k$--planes in $\bC^n$; here the group is $G_\bC\simeq\tSL_n\bC$ and $P$ is the stabilizer of a fixed $k$--plane.

Let $G_\bR$ be a (connected) real form of $G_\bC$.  There are only finitely many $G_\bR$--orbits on $\check D$.  An open $G_\bR$--orbit
\[
  D \ = \ G_\bR/V
\]
is a \emph{flag domain}.  The stabilizer $V \subset G_\bR$ is the centralizer of a torus $T' \subset G_\bR$, \cite[Corollary 2.2.3]{\FHW}.  When $D$ admits the structure of a Mumford--Tate domain, there exists a compact maximal torus $T \subset G_\bR$ such that $T' \subset T \subset V$.  We will assume this to be the case throughout.\footnote{In fact, if $D$ is a Mumford--Tate domain, then $V$ is compact.  However, we will not need this.}  In particular, 
\[
 \tdim_\bR T \ = \ \trank\,\fg_\bC \,.
\]

Throughout we identify $o \in D$ with both $V/V \in G_\bR/V$ and $P/P \in G_\bC/P$.

\subsection{Lie algebra structure}

Let $\ft \subset \fv \subset \fg_\bR$ be the Lie algebras of $T \subset V \subset G_\bR$.  Given a subspace $\fs \subset \fg_\bR$, let $\fs_\bC$ denote the complexification. Then $\fh = \ft_\bC$ is a Cartan subalgebra of $\fg_\bC$.  Let $\Delta = \Delta(\fg_\bC,\fh) \subset \fh^*$ denote the roots of $\fg_\bC$.  Given a root $\a \in \Delta$, let $\fg^\a \subset \fg_\bC$ denote the corresponding root space so that 
\begin{equation} \label{E:rtdecomp}
  \fg_\bC \ = \ \fh \ \op \ \bigoplus_{\a\in\Delta} \fg^\a \,.
\end{equation}
Since $T$ is compact, the roots $\a \in \Delta$ are pure imaginary on $\ft \subset \fh$.  Therefore,
\begin{equation} \label{E:conjga}
  \overline{\fg^\a} \ = \ \fg^{-\a} \,,
\end{equation}
where conjugation $\overline{\cdot}$ on $\fg_\bC$ is defined with respect to the real form $\fg_\bR$.  

Given any subspace $\fs \subset\fg_\bC$, let
$$
  \Delta(\fs) \ = \ \{ \a \in \Delta \ | \ \fg^\a \subset \fs \} \,.
$$
Given a subspace $\fs \subset \fg_\bR$, we will abuse notation by letting $\Delta(\fs)$ denote $\Delta(\fs_\bC)$.

The facts that $\fh = \ft_\bC \subset \fv_\bC$ and $[\fh , \fv_\bC] \subset \fv_\bC$ imply that 
$$
  \fv_\bC \ = \ \fh \ \op \ \bigoplus_{\a\in\Delta(\fv_\bC)} \fg^\a \,.
$$
As discussed above, $\fv_\bC$ is the centralizer of a subalgebra $\fh' = \ft'_\bC \subset \fh$.  Equivalently, 
$$
  \Delta(\fv_\bC) \ = \ \{ \a\in\Delta \ | \ \a(\fh') = 0 \} \,.
$$
In particular, 
\begin{equation} \label{E:Dv}
  -\Delta(\fv_\bC) \ = \ \Delta(\fv_\bC) \,.
\end{equation}

A choice of \emph{simple roots} $\Sigma = \{ \s_1 , \ldots , \s_r\} \subset \Delta$ is equivalent to a choice of positive roots $\Delta^+ \subset \Delta$.  A choice of positive roots $\Delta^+$ is equivalent to a choice of Borel subalgebra $\fb \supset \fh$ of $\fg_\bC$.  Our convention is that $\Delta(\fb) = \Delta^+$;  that is, 
\begin{equation} \label{E:b}
  \fb \ = \ \fh \ \op \ \bigoplus_{\a\in\Delta^+} \fg^\a \,.
\end{equation}
Define a parabolic subalgebra
\begin{equation} \label{E:dfn_p}
  \fp \ = \ \fv_\bC \ + \ \fb \,.
\end{equation}
By \eqref{E:conjga} and \eqref{E:Dv},
\begin{equation} \label{E:vC}
  \fp \,\cap\,\overline\fp \ = \ \fv_\bC \,.
\end{equation}

\subsection{Eigenspace decompositions}

Let $\{\ttS^1,\ldots,\ttS^r\}$ denote the basis of $\fh$ dual to the simple roots,
$$
  \s_i(\ttS^j) \ = \ \d^j_i \,.
$$  
Let 
$$
   I \ = \ I(\fv_\bC,\Sigma) 
   \ \dfn \  \{ i \ | \ \s_i \not\in \Delta(\fv_\bC)\} \\
   \ \stackrel{\eqref{E:vC}}{=} \ \{ i \ | \ -\s_i \not\in\Delta(\fp) \} \,.
$$
Then 
$$
 \fv_\bC \ = \ \fh' \ \op \ \fv_\bC^\tss \,,
$$
where $\fh' = \tspan_\bC\{ \ttS^i \ | \ i \in I \}$ is the center of $\fv_\bC$, and $\fv_\bC^\tss = [\fv_\bC , \fv_\bC]$ is the semisimple subalgebra with simple roots 
\begin{equation} \label{E:Sigmav}
  \Sigma(\fv_\bC) \ = \ \{ \s_i \ | \ i \not \in I\} \,.
\end{equation}
Define
\begin{equation} \label{E:ttE}
  \ttE \ \dfn \ \ttE(\fv_\bC,\Sigma) \ = \ 
  \sum_{i \in I} \ttS^i \,.
\end{equation}

\begin{remark}
The endomorphism $\ttE$ is a \emph{grading element}.  Grading elements may be viewed as infinitesimal Hodge structures, see \cite[Section 2.3]{ MR3217458} for a discussion.
\end{remark}

As an element of $\fh$, $\ttE$ is semisimple.  Therefore, every $\fg_\bC$--module decomposes into a direct sum of $\ttE$--eigenspaces.  Given a module $\sU$, let $\Lambda(\sU)$ denote the weights of $\sU$.  Then the $\ttE$--eigenvalues of $\sU $ are $\{ \lambda(\ttE) \ | \ \lambda \in \Lambda(\sU)\}$.  
If $\sU = \fg_\bC$, then $\Lambda(\sU) = \Delta$ and the eigenvalues are integers.  Let 
\begin{subequations} \label{SE:gell}
\begin{equation}
  \fg_\bC \ = \ \bigoplus_{\ell\in\bZ} \fg_\ell
\end{equation}
be the $\ttE$--eigenspace decomposition of $\fg_\bC$; explicitly,
\begin{equation}
  \fg_\ell \ = \ \{ X \in \fg_\bC \ | \ [\ttE , X] = \ell X \} \,.
\end{equation}
\end{subequations}
In terms of the root space decomposition \eqref{E:rtdecomp} of $\fg_\bC$, we have
\begin{eqnarray*}
  \fg_\ell & = & \bigoplus_{\a(\ttE)=\ell} \fg^\a \,,
  \quad \hbox{for } \ \ell\not=0 \,,\\
  \fg_0 & = & \fh \ \op \ \bigoplus_{\a(\ttE)=0} \fg^\a \,.
\end{eqnarray*}
Then \eqref{E:conjga} implies
\begin{equation} \label{E:conjgell}
  \overline{\fg}_\ell \ = \ \fg_{-\ell} \,.
\end{equation}
From \eqref{E:vC} and \eqref{E:conjgell} we see that
\begin{equation} \label{E:v=g0}
  \fv_\bC \ = \ \fg_0 \,.
\end{equation}

Let 
$$
  \fg_+ \ = \ \bigoplus_{\ell>0} \fg_\ell \tand
  \fg_- \ = \ \bigoplus_{\ell>0} \fg_{-\ell} \,.
$$
Then \eqref{E:dfn_p} implies
\begin{equation} \label{E:p}
  \fp \ = \ \fg_{\ge0} \ = \ \fg_0 \,\op\, \fg_+ \,.
\end{equation}


The Jacobi identity yields
\begin{equation} \label{E:grLB}
  [\fg_\ell,\fg_m] \ \subset \ \fg_{\ell+m} \,.
\end{equation}
The property \eqref{E:grLB} implies both $\fg_\pm$ are nilpotent, and that each 
\begin{equation} \label{E:gell}
  \hbox{$\fg_\ell$ is a $\fg_0$--module.}
\end{equation}  
The Killing form $B : \fg_\bC \times \fg_\bC \to \bC$ yields a $\fg_0$--module identification 
\begin{equation} \label{E:dual}
  \fg_\ell^* \ \simeq \ \fg_{-\ell} \,.
\end{equation}

\section{The infinitesimal period relation and characteristic cohomology} \label{S:IPR+CC}

\subsection{The infinitesimal period relation} \label{S:IPR}

The holomorphic tangent space at $o \in \check D$ is identified with $\fg_\bC/\fp$, as a $\fp$--module, and the \emph{holomorphic tangent bundle} is the $G_\bC$--homogeneous bundle
$$
  \cT \check D \ = \ G_\bC \times_P (\fg_\bC/\fp) \,.
$$
The equations \eqref{E:p} and \eqref{E:grLB} imply that $\fg_{\ge-1}/\fp$ is a $\fp$--module.  The homogeneous subbundle
\begin{equation} \nonumber 
  \cT_1 \ \dfn \ G_\bC \times_P (\fg_{\ge-1}/\fp)
\end{equation}
is the \emph{holomorphic infinitesimal period relation} on $\check D$.  

Let $T\check D$ denote the (real) tangent space, and $T_\bC \check D$ its complexification, so that 
$$
  \cT \check D \,\op\,\overline{\cT \check D} \ = \ T_\bC\check D \,.
$$
The \emph{complexified infinitesimal period relation} is 
$$
  T_{1,\bC} \ \dfn \ \cT_1 \,\op\,\overline{\cT_1} \ \subset \ T_\bC \check D \,.
$$
Finally, 
$$
  T_1 \ \dfn \ T_{1,\bC} \,\cap\, T\check D
$$
is the (real) \emph{infinitesimal period relation} (IPR).

A \emph{variation of Hodge structure} (VHS) is a solution of the IPR.  By this we mean either: (i) a connected complex submanifold $M \subset \check D$ with the property that $T M \subset \left. T^1\right|_{M}$; or (ii) irreducible variety $Y \subset \check D$ such that $T_yY \subset T_{1,y}$ for all smooth $y \in Y$.  (Equivalently, the smooth locus $M = Y^0$ is a solution in the first sense.)

\subsection{Bracket--generation} \label{S:BG}

The eigenspace decomposition \eqref{SE:gell} satisfies 
\begin{equation} \label{E:BG}
  \fg_{\ell+1} \ = \ [\fg_\ell,\fg_1] \tand
  \fg_{-\ell-1} \ = \ [\fg_{-\ell},\fg_{-1}] \quad\hbox{for any } \ \ell > 0\,,
\end{equation}
\cf \cite[Proposition 3.1.2]{MR2532439}.  Equivalently, the subbundles $T_1 \subset T D$ and $\cT_1 \subset \cT D$ are bracket--generating.

\begin{remark}
In general, the IPR, as it arises in Hodge theory, will not be bracket--generating.  However, for the purpose of studying the IPR, we may reduce to the case that it is, \cf\cite[Section 3.3]{ MR3217458}.
\end{remark}

\subsection{Characteristic cohomology} \label{S:dfnCC}

Given an open subset $U \subset \check D$, let $\cA_U$ denote the graded ring of smooth, complex--valued differential forms on $U$, and let $\cI_U \subset \cA_U$ be the graded, differential ideal generated by the smooth sections $\varphi : U \to \left.\tAnn(T_{1,\bC})\right|_U$ and their exterior derivatives $d \varphi$.  By construction $\cI_U$ is differentially closed:
$$
  \td\cI_U \ \subset \ \cI_U \,.
$$
Whence the de Rham complex $(\cA_U,\td)$ induces a quotient complex $(\cA_U/\cI_U , \td )$.  The \emph{characteristic cohomology} of the IPR on $U \subset \check D$ is the associated cohomology
$$
  H_{\cI}^\sb(U) \ \dfn \ H^\sb( \cA_U/\cI_U , d ) \,.
$$
Note that $M \subset U$ is a VHS if and only if $\left.\cI_U\right|_M = 0$.  (For this reason, we also call $\cI$ the \emph{infinitesimal period relation}.)  Therefore, the characteristic cohomology pulls--back to de Rham cohomology on $M$; that is, there exists a natural map $H_\cI^\sb(U) \to H^\sb(M,\bC)$.  This is the sense in which the characteristic cohomology induces ordinary cohomology on solutions.

\section{Characteristic cohomology on the compact dual} \label{S:CCcD}

In this section we consider the global characteristic cohomology; that is, we fix $U = \check D$.  Through out this section we simplify notation by writing $\cA$ and $\cI$ for $\cA_{\check D}$ and $\cI_{\check D}$, respectively.  We will see that the Schubert varieties $X_w \subset\check D$ and their homology classes $\bx_w \in H_\sb(\check D , \bZ)$ play a key r\^ole here.  The terminology \emph{Schubert VHS} indicates a Schubert variety that is also a VHS (Section \ref{S:IPR}).  The three main results of this section are as follows: First, the characteristic cohomology is spanned by the cohomology classes dual to the Schubert VHS (Theorem \ref{T:CCcD}).  Second, a homology class $\mathbf{y} \in H_\sb(\check D, \bZ)$ may be represented by a union $Y_1\cup\cdots\cup Y_s$ of VHS if and only if it may be represented by a union of Schubert VHS (Theorem \ref{T:hom}).  As a corollary to these two theorems, we obtain the third result, an $\cI$--de Rham theorem (Corollary \ref{C:deRham}).  Schubert varieties and the characterization of Schubert VHS are briefly reviewed in Sections \ref{S:schub} and \ref{S: MR3217458}.

\subsection{Schubert varieties} \label{S:schub}

This section does little more than establish notation for our discussion of Schubert varieties.  The reader interested in greater detail is encouraged to consult \cite{ MR3217458} and the references therein.

Given simple root $\s_i\in\Sigma$, let $(i) \in \tAut(\fh^*)$ denote the corresponding \emph{simple reflection}.  The \emph{Weyl group} $W \subset \tAut(\fh^*)$ of $\fg_\bR$ is the group generated by the simple reflections $\{ (i) \ | \ \s_i\in\Sigma \}$.  A composition of simple reflections $(i_1) \circ (i_2) \circ\cdots\circ(i_t)$, which are understood to act on the left, is written $(i_1 i_2 \cdots i_t) \in W$.  The \emph{length} of a Weyl group element $w$ is the minimal number 
$$
  |w| \ \dfn \ \tmin\{ \ell \ | \ w = (i_1 i_2 \cdots i_\ell) \}
$$
of simple reflections necessary to represent $w$.

Let $W_\fp \subset W$ be the subgroup generated by the simple reflections $\{(i) \ | \ i \not \in I\}$.  Then $W_\fp$ is naturally identified with the Weyl group of $\fg^0$.  The rational homogeneous variety $G/P$ decomposes into a finite number of $B$--orbits 
$$
  G/P \ = \ \bigcup_{W_\fp w \in W_\fp\backslash W} B w^{-1} o 
$$
which are indexed by the right cosets $W_\fp\backslash W$.  The \emph{$B$--Schubert varieties} of $G/P$ are the Zariski closures
$$
  X_w \ \dfn \ \overline{B w^{-1} o }\,.
$$

Let 
$$
  \bx_w \ \dfn \ [X_w] \ \in \ H_\sb(\check D, \bZ)
$$
denote the homology class represented by the Schubert variety.  Borel \cite{MR0077878} showed that the Schubert classes form a free additive basis of the integral homology
$$
  H_\sb (\check D , \bZ) \ = \ \tspan_\bZ\{ \bx_w \ | \ w \in W^\fp \} \,.
$$
Since $G_\bC$ is path connected, any $G_\bC$--translate $g X_w$ satisfies $[g X_w] = \bx_w$.  We will refer to any of these translates as a Schubert variety (of type $W_\fp w$).  

Each right coset $W_\fp\backslash W$ admits unique representative of minimal length; let 
$$
  W^\fp \ \simeq \ W_\fp \backslash W
$$ 
be the set of minimal length representatives.  (See Appendix \ref{S:kostant} for a terse discussion of how $W^\fp$ is determined.)  For a minimal representative $w \in W^\fp$, the Schubert variety $w X_w$ is the Zariski closure of $N_w \cdot o$, where $N_w \subset G$ is a unipotent subgroup with nilpotent Lie algebra 
\begin{equation} \label{E:nw}
  \fn_w \ \dfn \ \bigoplus_{\a\in \Delta(w)} \fg^{-\a} \ \subset \ \fg^{-}
\end{equation}
given by 
\begin{equation} \label{E:Dw}
  \Delta(w) \ \dfn \ \Delta^+ \,\cap \, w(\Delta^-) \,.
\end{equation}
Moreover, $N_w \cdot o$ is an affine cell isomorphic to $\fn_w$, and $\tdim\,X_w = \tdim\,\fn_w = |\Delta(w)|$.  Indeed 
$$
  T_o X_w \ = \ \fn_w \,.
$$
For any $w \in W^\fp$ we have 
\begin{equation}\label{E:len=dim}
  |w| \ = \ |\Delta(w)| \ = \ \tdim\,X_w \,.
\end{equation}

\subsection{Schubert VHS} \label{S: MR3217458}

A Schubert variety $X_w$ is a VHS if and only if $\Delta(w) \subset \Delta(\fg_1)$, where $\Delta(w)$ is given by \eqref{E:Dw}, \cf\cite[Theorem 3.8]{ MR3217458}.  A convenient way to test for this condition is as follows.  Let
$$
  \rho \ \dfn \ \sum_{i=1}^r \w_i \ = \ \half \sum_{\a\in\Delta^+} \a
$$
be the sum of the fundamental weights (which is also half the sum of the positive roots).  Define
\begin{equation} \label{E:dfn_rho}
  \rho_w \ \dfn \ \rho \,-\, w(\rho) \ = \ \sum_{\a\in\Delta(w)} \a \,.
\end{equation}
(See \cite[(5.10.1)]{MR0114875} for the second equality.)  Then 
$$
  |w| \ \le \ \rho_w(\ttE) \ \in \ \bZ \,,
$$
and equality holds if and only if $\Delta(w) \subset \Delta(\fg_1)$; equivalently, $X_w$ is a variation of Hodge structure if and only if $\rho_w(\ttE) = |w|$.  See \cite[Section 3.5]{ MR3217458} for details.  Let
$$
  W_\mathrm{vhs} \ \dfn \ \{ w \in W^\fp \ | \ \rho_w(\ttE) = |w| \} 
$$
be the set indexing the Schubert variations of Hodge structure.\footnote{The sets $W_\mathrm{vhs} \subset W^\fp$ are denoted by $W^\varphi_\sI \subset W^\varphi$ in \cite{ MR3217458}.}  

\subsection{Characteristic cohomology} \label{S:schubCC}

Let $\bx^w \in H^\sb(\check D,\bZ)$ denote the cohomology classes dual to the Schubert classes $\bx_w$ (Section \ref{S:schub}).  Roughly, the following theorem asserts that the characteristic cohomology is spanned by the classes dual to the Schubert VHS.  

\begin{theorem}\label{T:CCcD}
Let $p_\cI : H^\sb(\check D) \to H^\sb_\cI(\check D)$ be the ring homomorphism induced by the natural map $(\cA,\td) \to (\cA/\cI , \td)$ of complexes.  Then $p_\cI$ is surjective and 
\[
  \tker\,p_\cI \ = \ \tspan\{ \bx^w \ | \ w \in W^\fp\backslash W_\mathrm{vhs} \} \,.
\]
In particular, the map $p_\cI$ is given by 
\[
  \bc \ = \ \sum_{w \in W^\fp} c_w \bx^w \quad \mapsto \quad
  \bc_\cI \ \equiv \ \sum_{w \in W_\mathrm{vhs}} c_w \bx^w \,.
\]
Thus, $H^\sb_\cI(\check D) \equiv \tspan\{ \bx^w \ | \ w \in W_\mathrm{vhs}\}$. 
\end{theorem}

\noindent Above, we use $\equiv$ (in place of $=$) to emphasize that $\bc_\cI \in H^\sb(\check D) / \tker\,p_\cI$.

\begin{proof}
Given \cite[(4.5)]{ MR3217458}, this follows from the same arguments in \cite[Sections 4.1.3--4.1.5]{ MR3217458} which establish \cite[Theorem 4.1]{ MR3217458}.
\end{proof}

\subsection{Homology of VHS} \label{S:hom}

We next identify the homology classes $\by \in H_\sb(\check D,\bZ)$ that may be represented by a union of VHS.  First, by Borel's result (Section \ref{S:schubCC}), the homology class represented by a subvariety $Y \subset G_\bC/P$ is a linear combination of the form
\begin{equation} \label{E:[Y]}
  [Y] \ = \ \sum_{w \in W^\fp} n^w \bx_w \,,
\end{equation}
with \emph{nonnegative} coefficients $0 \le n^w \in \bZ$.  We will show that a homology class may be represented by a (union of) VHS if and only if it may be represented by a union of Schubert VHS.

\begin{theorem} \label{T:hom}
A homology class $\by \in H_\sb(\check D,\bZ)$ may be represented by a union of VHS if and only if 
\begin{equation} \label{E:hom}
  \by \ = \ \sum_{w\in\Wvhs} n^w \bx_w \quad\hbox{with } \ 0 \le n^w \in \bZ \,.
\end{equation}
\end{theorem}

The \emph{$\cI$--homology of the IPR} is the homology
$$
  H_{\sb,\cI}(\check D) \ = \ \tspan\{ [Y] \in H_\sb(\check D) \ | \ Y \hbox{ is a VHS}\}\,.
$$
From Theorems \ref{T:CCcD} and \ref{T:hom} we obtain

\begin{corollary}[The $\cI$--de Rham theorem for the compact dual] \label{C:deRham}
The Poincar\'e pairing 
$$
  H_{\sb,\cI}(\check D) \,\times\, H^\sb_\cI(\check D) \ \to \ \bC
$$
is nondegenerate.
\end{corollary}

\begin{proof}[Proof of Theorem \ref{T:hom}]
Of course the implication $(\Longleftarrow)$ is trivial: given \eqref{E:hom}, the homology class $\by$ is represented by 
$$
  Y \ = \ \sum_{w \in \Wvhs} n^w X_w\,. 
$$

For the converse $(\Longrightarrow)$ we may assume that $\by = [Y]$ with $Y$ an irreducible VHS.  The coefficients of \eqref{E:[Y]} are given by 
\begin{equation} \label{E:prf}
  n^w \ = \ \int_Y \bx^w  \,,
\end{equation}
with $|w|$ the (complex) dimension of $Y$. Recall (Section \ref{S:dfnCC}) that a subvariety $Y \subset \check D$ is a VHS if and only if $\cI$ vanishes when pulled-back to the smooth locus of $Y$.  Suppose that $w \in W^\fp \backslash \Wvhs$ indexes a Schubert variety that is \emph{not} a VHS.  Then $\bx^w$ admits a representative that is contained in the ideal $\cI$ (Lemma \ref{L:hom}).  Whence, \eqref{E:hom} follows from \eqref{E:prf} and the hypothesis that $Y$ is a VHS.
\end{proof}

\begin{lemma}  \label{L:hom}
The cohomology class $\bx^w$ admits a representative (which we may take to be invariant with respect to a compact real form $K$ of $G_\bC$) that is contained in the ideal $\cI$ if and only if $w \in W^\fp \backslash \Wvhs$ indexes a Schubert variety that is \emph{not} a VHS.
\end{lemma}

\begin{proof}
Suppose that the cohomology class $\bx^w$ admits a representative $\phi \in \cI$.  Then $\phi$ vanishes on every VHS.  In particular, $\phi$ vanishes on $X_v$ for all $v \in \Wvhs$.  Since $\phi$ does not vanish on the Schubert variety $X_w$, it follows that $w \not\in \Wvhs$ and $X_w$ is not a VHS.

The converse is a consequence of Kostant's \cite{MR0142697} and the description of the Schubert VHS in Section \ref{S:schub}.  Kostant exhibits a $K$--invariant differential form $\w^w$ representing a (positive) multiple of the class $\bx^w$, \cf\cite[Theorem 6.15]{MR0142697}.  Let $s^w = \w^w_o$ denote the form at $o \in \check D$.  Then a formula for $s^w$ is given by \cite[Theorem 5.6]{MR0142697}.  From this formula we see that $\w^w \in \cI$ if and only if $w \in W^\fp \backslash \Wvhs$.  So, if $X_w$ is not a VHS, then $\w^w \in \cI$.
\end{proof}

\section{A double complex on the flag domain} \label{S:CCD}

The main result of this section is the identification of the characteristic cohomology $H^\sb_\cI(D)$ with the total cohomology of a double complex of $G_\bR$--invariant differential operators (Theorem \ref{T:cc}).  The fact that the characteristic cohomology can be realized as cohomology on a complex of vector bundles over $D$ is well--understood, \cf\cite{DanielMa}; the significance of Theorem \ref{T:cc} is that it gives an explicit, representation theoretic description of the $G_\bR$--homogeneous vector bundles in the double complex.  This provides the information necessary to prove the results in Section \ref{S:cvc} relating the characteristic cohomology to the de Rham cohomology.

\subsection{$G_\bR$--homogeneous bundles on $D$} \label{S:onD}

Recall (Section \ref{S:IPR}) that the holomorphic tangent space $\cT_oD \simeq \cT_o\check D \simeq \fg_\bC/\fp$ as a $\fp$--module.  It follows from \eqref{E:v=g0} and \eqref{E:p} that $\cT_o D \simeq \fg_-$ as a $V$--module.  Therefore, the holomorphic tangent bundle of $D$ is the $G_\bR$--homogeneous vector bundle
\begin{equation}\label{E:dfn_cT}
  \cT D \ = \ G_\bR \times_V \fg_- \,.
\end{equation}
Likewise, the tangent bundle is a $G_\bR$--homogeneous vector bundle, described as follows.  By \eqref{E:conjgell} and \eqref{E:v=g0}, 
\begin{equation} \nonumber 
  \fv^\perp \ = \ (\fg_- \op \fg_+) \,\cap\, \fg_\bR
\end{equation}
is a real form of $\fg_- \op \fg_+$.  In particular, 
$$
  \fg_\bR \ = \ \fv \ \op \ \fv^\perp
$$
is a $V$--module decomposition.  So the tangent space $T_oD$ is naturally identified with $\fg_\bR/\fv = \fv^\perp$, as a $V$--module.  Moreover, the (real) tangent bundle $TD$ is the $G_\bR$--homogeneous bundle
$$
  TD \ = \ G_\bR \times_V \fv^\perp \,.
$$
Given $\ell > 0$, \eqref{E:conjgell} implies the subspace 
\begin{equation} \nonumber 
  \fv^\perp_\ell \ = \ (\fg_\ell \op \fg_{-\ell}) \,\cap\, \fg_\bR
\end{equation}
is a real form of $\fg_\ell \op \fg_{-\ell}$.  Additionally, \eqref{E:v=g0} and \eqref{E:gell} imply that $\fv^\perp_\ell$ is a $V$--module.   So, for $\ell > 0$, we may define homogeneous sub-bundles 
$$
  T_\ell\ = \ G_\bR \times_V \fv^\perp_\ell \,.
$$
Note that $TD = \op_\ell \,T_\ell$.  The complexified tangent bundle is the $G_\bR$--homogeneous bundle
$$
  T_\bC D \ = \ G_\bR \times_V (\fg_- \op \fg_+) \,.
$$
We have 
\begin{equation} \label{E:TC_decomp}
  T_\bC D \ = \ \bigoplus_{0 < \ell} T_{\ell,\bC} \,,
\end{equation} 
where $T_{\ell,\bC} = G_\bR \times_V (\fg_{-\ell} \op \fg_{\ell})$ is the complexification of $T_\ell$.  

The complexified cotangent bundle is
$$
  T^*_\bC D \ = \ G_\bR \times_V (\fv^\perp_\bC)^* 
  \ \simeq \ \oplus_\ell\, (T_{\ell,\bC})^*  \,,
$$
Let $\tAnn(\fv^\perp_{1,\bC}) \subset (\fv^\perp_\bC)^*$ denote the annihilator of $\fv^\perp_{1,\bC}$.  Then the annihilator of $T_{1,\bC}$ is
\begin{equation} \label{E:AnnT1}
  \tAnn(T_{1,\bC}) \ = \ G \times_V \tAnn(\fv^\perp_{1,\bC}) \,.
\end{equation}

Let 
$$
  \tw^k_D \ \dfn \ \tw^k T^*_\bC D \ = \ 
  G_\bR \times_V \tw^k (\fv^\perp_\bC)^* 
$$
denote the $k$--th exterior power, so that $\cA^k_D$ is the space of smooth sections of $\tw^k_D$.  Define $G_\bR$--homogeneous bundles
\begin{equation} \label{E:dfn_pq}
  \tw^{p,q}_D \ \dfn \
  G_\bR \times_V \left( \tw^p \fg_-^*)\ot (\tw^q \fg_+^*) \right)
  \ \simeq \ 
  G_\bR \times_V \left( \tw^p \fg_+)\ot (\tw^q \fg_-) \right) \,.
\end{equation}
Note that
\begin{equation}\label{E:cT*}
  \cT^* D \ = \ \tw^{1,0}_D \tand
  \overline{\cT^* D} \ = \ \tw^{0,1}_D \,,
\end{equation}
and
$$
  \tw^k_D \ = \ \bigoplus_{p+q=k} \tw^{p,q}_D
$$
as $V$--modules.  Given an open subset $U \subset D$, let $\cA^{p,q}_U$ denote the smooth, complex--valued sections $U \to \tw^{p,q}_D$; that is, $\cA^{p,q}_U$ is the space of smooth, complex--valued $(p,q)$--forms on $U$.  We have 
\[
  \td \ = \ \partial + \bar\partial
\]
with
\[
  \partial : \cA^{p,q}_U \,\to\, \cA^{p+1,q}_U
  \tand
  \bar\partial : \cA^{p,q}_U \,\to\, \cA^{p,q+1}_U\,.
\]

\subsection{Outline of the proof of Theorem \ref{T:cc}} \label{S:outline}

For the remainder of Section \ref{S:CCD} we simplify notation by writing $\cA$ and $\cI$ for $\cA_D$ and $\cI_D$, respectively.  Recall (Section \ref{S:dfnCC}), that $\cI$ is the differential ideal generated by the smooth sections of \eqref{E:AnnT1}.  In Section \ref{S:cI} we will show that the ideal $\cI$ is the space of sections of a homogeneous sub-bundle $I \subset \tw^\sb_D$.  From the structure of the bundle $I$ we will obtain Theorem \ref{T:cc}, which asserts that the characteristic cohomology may be realized as the cohomology of the total complex $(\cC^\sb,\mathbf{d})$ associated with a double complex $(\cC^{\sb,\sb} , \d,\bar\d)$ of $G_\bR$--invariant differential operators.  The theorem is proved in Sections \ref{S:cI}--\ref{S:Iperp}.  

Before launching into the details of the proof, I will sketch the argument.  First, we show that there exists a $G_0$--submodule $\ffi \subset \tw^\sb(\fv^\perp_{\bC})^*$ such that $\cI$ is the space of smooth sections of the homogeneous subbundle $I = G_\bR \times_V \ffi \subset \tw^\sb_D$, \cf\eqref{E:i}. 

Since $V$ is reductive, there exists a $V$--submodule $\ffi^\perp$ such that $\tw^\sb(\fv^\perp_{\bC})^* = \ffi \op \ffi^\perp$.  Let $\cC \subset \cA$ be the smooth sections of the homogeneous bundle $I^\perp = G_\bR\times_V \ffi^\perp$.  The decomposition $\tw^\sb D = I \op I^\perp$ then yields a natural projection
\begin{equation} \label{E:A-C}
  \wp \,:\, \cA \ \to \ \cC \,,
\end{equation}
and 
\begin{equation} \label{E:A/I=C}
  \cA/\cI \ \simeq \ \cC  \,.
\end{equation}

Second, a detailed description of the $V$--module structure of $\ffi^\perp$ will imply that $\cC$ inherits a bigrading from $\cA^{\sb,\sb}$.  That is,
\begin{equation} \label{E:Cpq}
\renewcommand{\arraystretch}{1.3}
\begin{array}{rcl}
  \cC \ = \  \op\,\cC^k \,, & \hbox{ where } & \cC^k \ = \ \cC \cap \cA^k \,, 
  \quad \hbox{and}\\
  \cC^k \ = \  \op_{p+q=k}\,\cC^{p,q} \,, & \hbox{ with } & 
  \cC^{p,q} \ = \ \cC^k \cap \cA^{p,q} \tand \overline{\cC^{p,q}} \,=\, \cC^{q,p} \,.
\end{array}
\end{equation}
Let
\begin{eqnarray}
  \nonumber
  \mathbf{d} & = & \wp \circ \td : \cC^k \to \cC^{k+1} \,,\\
  \label{E:d}
  \d & = &  \wp \circ \del : \cC^{p,q} \to \cC^{p+1,q} \,,\\
  \nonumber
  \bar\d & = &  \wp \circ \bar\del : \cC^{p,q} \to \cC^{p,q+1} \,.
\end{eqnarray}
Clearly, $\mathbf{d} = \d + \bar\d$.  Additionally, $\td \cI \subset \cI$ implies $0 = \mathbf{d}^2$, so that $0 = \d^2 = \bar\d{}^2 = \d\,\bar\d + \bar\d\,\d$.  Since $\td$, $\del$ and $\bar\del$ are $G_\bR$--invariant differential operators, and the projection $\wp$ is a $G_\bR$--module map, it follows that $(\cC^{\sb,\sb} , \d,\bar\d)$ is a bigraded complex of $G_\bR$--invariant differential operators.  Finally, \eqref{E:A/I=C} identifies the complex $(\cA/\cI,\td)$ defining the characteristic cohomology with the total complex $(\cC^\sb, \bd)$.  Thus,
$$
  H^\sb_\cI(D) \ = \ H^\sb(\cC,\bd) \,.
$$
More generally, $H^\sb_\cI(U) = H^\sb(\cC_U,\bd)$ for any open set $U \subset D$; though the differential operators $\bd, \del,\bar\del$ are no longer $G_\bR$--equivariant when restricted to $U \subsetneq D$ (because $G_\bR$ does not preserve $U$).

We now proceed with the details.  

\subsection{\bmath{The ideal $\cI$ as sections of a homogeneous sub-bundle}} \label{S:cI}

Let $\cI_1 \subset \cA$ be the graded ideal generated by the smooth sections of $\tAnn(T_{1,\bC})$.  Then
$$
  \cI \ = \ \cI_1 \,+\, \td\cI_1 \,.
$$  
Observe that the ideal $\ffi_1 \subset \tw^\sb(\fv^\perp_{\bC})^*$ generated by $\tAnn(\fv^\perp_{1,\bC})$ is a $V$--module.  From \eqref{E:AnnT1} we see that the ideal $\cI_1$ is naturally identified with the smooth sections of 
$$
  I_1 \ \dfn \ G_\bR \times_V \ffi_1 \,.
$$
It remains to account for $\td\cI_1$ modulo $\cI_1$.

\begin{remark}[Conventions] \label{R:conven}
Throughout we will regard $(\fv^\perp_{1,\bC})^*$ as a subspace of $(\fv^\perp_{\bC})^*$ by identifying it with the annihilator of $\op_{\ell\ge2}\fv^\perp_{\ell,\bC}$.  Then, by extension, we will regard 
$$
  \ffi_1^\perp \ \dfn \ \tw^\sb(\fv^\perp_{1,\bC})^*
$$ 
as a subspace of $\tw^\sb(\fv^\perp_{\bC})^*$.  Under this identification 
\begin{equation} \label{E:conven}
  \tw^\sb(\fv^\perp_{\bC})^* \ = \ \ffi_1 \ \op \ \ffi_1^\perp \,
\end{equation}
is a $V$--module decomposition.
\end{remark}

\begin{claim*} There is a $V$--module inclusion $(\fv^\perp_{2,\bC})^* \inj \tw^2(\fv^\perp_{1,\bC})^*$.
\end{claim*}

\begin{proof}
To see this, let $\xi \in (\fv^\perp_{2,\bC})^* = (\fg_{-2} \op \fg_2)^*$ and $x,y \in \fv^\perp_{1,\bC} = \fg_{-1}\op\fg_1$.  Then $[x,y] \subset \fg_{-2}\op\fg_0\op\fg_2$ by \eqref{E:grLB}.  Thus, $\xi(x,y) = \xi([x,y])$ defines a $V$--module map $(\fv^\perp_{2,\bC})^* \to \tw^2(\fv^\perp_{1,\bC})^*$.  In fact, 
\begin{equation} \label{E:image}
\hbox{the image of $\fg_{\pm2}^*$ under $(\fv^\perp_{2,\bC})^* \to \tw^2(\fv^\perp_{1,\bC})^*$ lies in $\tw^2\fg_{\pm1}^*$.}
\end{equation}
It follows from \eqref{E:BG} that $(\fv^\perp_{2,\bC})^* \to \tw^2(\fv^\perp_{1,\bC})^*$ is injective.  
\end{proof}

Let 
\[
  T' \ \subset \ \tw^2 D
\] 
be the corresponding $G_\bR$--homogeneous sub-bundle.  Let $\cC^\infty( T')$ denote the space of smooth sections.  We will show that
\begin{equation} \label{E:dAnn}
  \td \, \cC^\infty( \tAnn(T_{1,\bC})) \ \equiv \ \cC^\infty( T') \quad \tmod\quad
  \cI_1 \,.
\end{equation}
First we note some consequences of the equation.  Let $\cI' \subset \cA$ be the ideal generated by the smooth sections of $T'$.  Then
\begin{equation} \label{E:genI}
  \cI \ = \ \cI_1 \,+\, \cI' \,.
\end{equation}
Let $\ffi' \subset \tw^\sb(\fv^\perp_{1,\bC})^*$ be the ideal generated by $(\fv^\perp_{2,\bC})^* \inj \tw^2(\fv^\perp_{1,\bC})^*$.  By \eqref{E:conven}
\begin{equation}\label{E:i'}
  \ffi \ \dfn \ \ffi_1 \,\op\, \ffi'
\end{equation}
is a direct sum.  Note also that $\ffi$ is an ideal of $\tw^\sb(\fv^\perp_{\bC})^*$.  Let $I = G\times_V \ffi \subset \tw^\sb D$ be the corresponding homogeneous vector bundle.  
\begin{equation} \label{E:i}
\hbox{\emph{The ideal $\cI$ is the space of smooth sections of $I$.}}
\end{equation}

\begin{proof}[Proof of \eqref{E:dAnn}]
Let $\varphi \in \cC^\infty( \tAnn(T_{1,\bC}))$, and let $X,Y$ be smooth complex vector fields (sections of $T_\bC D$).  Then 
\begin{equation} \label{E:d1}
  d \varphi(X,Y) \ = \ X \varphi(Y) \,-\, Y \varphi(X) \,-\, \varphi([X,Y]) \,.
\end{equation}
Since we are computing $d \varphi$ modulo $\cI_1$, we may assume that $X,Y$ are sections of $T_{1,\bC}$.  Since $\varphi$ annihilates $T_{1,\bC}$, we have $\varphi(X) = \varphi(Y) = 0$.  Moreover, \eqref{E:grLB} and the definition of $T_{\ell,\bC}$ (Section \ref{S:onD}) imply $[X,Y]$ is a section of $T_{1,\bC}\op T_{2,\bC}$.  Let $[X,Y]_2$ denote the component of $[X,Y]$ taking values in $T_{2,\bC}$.  Again, since $\varphi$ annihilates $T_{1,\bC}$, we have $\varphi([X,Y]) = \varphi([X,Y]_2)$.  These observations, along with \eqref{E:d1}, yield
\begin{equation} \label{E:d2}
  d\varphi(X,Y) \ = \ -\varphi([X,Y]_2) \,.
\end{equation}

Note that every element $\psi \in \cC^\infty(T_{\bC}')$ is of the form $\psi(X,Y) = \psi_o([X,Y])$ where $\psi_o \in \cA^1$ is a 1-form annihilating $T_{\ell,\bC}$ for all $\ell\not=2$.  Equation \eqref{E:d2} asserts that $d\varphi = \psi$, modulo $\cC^\infty(\tAnn(T_{1,\bC}))$, where $\psi_o$ is defined by $\left.\psi_o\right|_{T_{2,\bC}} = -\left.\varphi\right|_{T_{2,\bC}}$.  This establishes the containment $\subset$ in \eqref{E:dAnn}.  Conversely, $\psi \equiv -\td \psi_o$ modulo $\cC^\infty(\tAnn(T_{1,\bC}))$.  This establishes \eqref{E:dAnn}.
\end{proof}

\subsection{\bmath{The complimentary sub-module $\ffi^\perp \subset \tw^\sb (\fg_-\op\fg_+)^*$}}

Since $G_0 = V_\bC$ is reductive and $\ffi \subset \tw^\sb (\fv^\perp_{\bC})^*$ is a $V$--submodule, there exists a $V$--module $\ffi^\perp$ such that 
\begin{equation} \label{E:i+iperp}
    \ffi \ \op \ \ffi^\perp \ = \ \tw^\sb (\fv^\perp_{\bC})^* \,.
\end{equation}
Assertions \eqref{E:A-C} and \eqref{E:A/I=C} of the outline (Section \ref{S:outline}) now follow.  The second step towards Theorem \ref{T:cc} is to identify the complement $\ffi^\perp$.  From \eqref{E:conven} and \eqref{E:i'} we see that $\ffi^\perp \subset \ffi_1^\perp = \tw^\sb(\fv^\perp_{1,\bC})^*$, and 
$$
  \ffi' \ \op\ \ffi^\perp \ = \ \tw^\sb (\fv^\perp_{1,\bC})^* \,.
$$

By \eqref{E:image}, $\fg_{-2}^* \inj \tw^2\fg_{-1}^*$.  Let $\fj \subset \tw^\sb \fg_{-1}^*$ denote the ideal generated by $\fg_{-2}^* \subset \tw^2\fg_{-1}^*$.  Note that $\fj$ is a homogeneous graded ideal; precisely, $\fj = \op\,\fj^\ell$ where $\fj^\ell = \fj \cap \tw^\ell\fg_{-1}^*$.  Equation \eqref{E:conjgell} implies that the conjugate $\overline{\fj} \subset \tw^\sb\fg_{1}^*$ is the ideal generated by $\fg_{2}^* \subset \tw^2\fg_{1}^*$.  Note that both $\fj$ and $\overline\fj$ are $V$--modules.  Moreover, \eqref{E:image} implies that the homogeneous component $(\ffi')^k$ of $\ffi'$ in 
$$
  \tw^k(\fv^\perp_{1,\bC})^* \ \simeq \ \bigoplus_{p+q=k}
  \big( \tw^p \fg_{-1}^* \big) \,\ot\, \big( \tw^q\fg_1^* \big) \,.
$$
is 
$$
  (\ffi')^k \ \simeq \ \sum_{p+q=k} 
  \big( \fj^p \,\ot\, \tw^q\fg_1^* \big) \,+\,
  \big( \tw^p \fg_{-1}^* \,\ot\, \overline{\fj}{}^q \big) \,.
$$
(The latter is not a direct sum, as the distinct summands may have nontrivial intersections.)  In particular, $\ffi' \simeq ( \fj \ot  \tw^\sb\fg_1^*) +  (\tw^\sb \fg_{-1}^* \ot \overline{\fj})$.  Therefore, if $\fj^\perp \subset \tw^\sb \fg_{-1}^*$ is a $V$--module complement to $\fj$, then    
\begin{equation} \label{E:iperp}
  \ffi^\perp \ = \ \fj^\perp \,\ot\, \overline{\fj^\perp} \,.
\end{equation}
The submodule $\fj^\perp$ is identified in \cite{ MR3217458} using Kostant's theorem on Lie algebra cohomology.  

\subsection{Lie algebra cohomology} \label{S:lac}

Lie algebra cohomology was introduced by Chevalley and Eilenberg \cite{MR0024908}.  Given a Lie algebra $\fa$ defined over $\bC$ define $\e : \tw^\ell \fa^* \to \tw^{\ell+1}\fa^*$ by 
\begin{equation}\label{E:lac_d}
  (\e\phi)(A_0 , \ldots , A_k) \ \dfn \ 
  \sum_{i<j}  (-1)^{i+j} \phi\left( [ A_i , A_j ] , 
  A_0 , \ldots , \hat A_i , \ldots , \hat A_j , \ldots , A_\ell \right)
\end{equation}
for any $\phi \in \tw^\ell \fa^*$ and $(\ell+1)$--tuple $A_0 , \ldots , A_\ell \in \fa$.  It is straightforward to confirm that $\e^2 = 0$.  Let
\begin{equation} \label{E:lac}
  H^\ell(\fa , \bC) \ = \ 
  \frac{\tker\{ \e : \tw^\ell \fa^* \to \tw^{\ell+1}\fa^* \}}
       {\tim\{ \e : \tw^{\ell-1} \fa^* \to \tw^\ell\fa^*  \}}
\end{equation}
denote the corresponding \emph{Lie algebra cohomology} (with coefficients in the trivial representation).

If $\fa = \fg_\pm$, then $\e$ is a $G_0$--module map, and $H^\sb(\fg_\pm,\bC)$ is a $G_0$--module.  Since $\ttE \in \fg_0$ is semisimple, it follows that the cohomology decomposes into $\ttE$--eigenspaces.  From the definition \eqref{E:lac}, we see that the $\ttE$--eigenvalues of $H^\ell(\fg_-,\bC)$ are integers $\ge \ell$; that is,
\begin{equation}\label{E:Hev}
  H^\ell(\fg_-,\bC) \ = \ 
  H^\ell_\ell\,\op\,H^\ell_{\ell+1}\,\op\,H^\ell_{\ell+2}\,\op\cdots
\end{equation}
where $H^\ell_m \subset H^\ell(\fg_-,\bC)$ is the $\ttE$--eigenspace with $\ttE$--eigenvalue $m$.\footnote{Examples of the eigenspace decomposition \eqref{E:Hev} are given in Appendix \ref{S:egs}.}  In \cite[\S4.2]{ MR3217458} it is shown that $H^\ell_\ell$ is the $V$--module complement to $\fj^\ell$ in $\tw^\ell \fg_{-1}^*$, and 
\begin{equation} \label{E:jperp}
  \fj^\perp \ = \ \bigoplus_{\ell\ge0} H_\ell^\ell \,.
\end{equation} 

Before continuing with the proof of Theorem \ref{T:cc}, we make two observations that will be useful later.  First, \eqref{E:conjgell} and \eqref{E:dual} imply that 
\begin{equation} \label{E:Hconj}
  H^\sb(\fg_+,\bC) \ = \ \overline{H^\sb(\fg_-,\bC)} \ = \ H^\sb(\fg_-,\bC)^*
\end{equation}
and the $\ttE$--eigenvalues of $H^\ell(\fg_+,\bC)$ are $-\ell, -\ell-1, -\ell-2, \ldots$  Second, \begin{equation} \label{E:H1}
  H^1(\fg_-,\bC) \ = \ H^1_1 \,;
\end{equation}
equivalently, $H^1_m = 0$ if $m>1$.  This is a consequence of Kostant's description \cite[Theorem 5.14]{MR0142696} of the $G_0$--module structure of $H^\sb(\fg_-,\bC)$.  Given $i \in I$, let $H_{(i)}$ be the irreducible $G_0$--module of highest weight $\s_i$.  Then Kostant's theorem asserts that
$$
  H^1(\fg_-,\bC) \ = \ \bigoplus_{i \in I} H_{(i)} \,.
$$
Since $H_{(i)}$ is irreducible, and $\ttE$ lies it the center of the reductive $\fg_0$, $\ttE$ necessarily acts by a scalar, which must be $\s_i(\ttE) = 1$ by \eqref{E:ttE}.  Thus \eqref{E:H1} holds.

\subsection{\bmath{The complimentary sub-bundle $I^\perp \subset \tw^\sb D$}} \label{S:Iperp}

Equations \eqref{E:iperp} and \eqref{E:jperp} yield
\begin{equation} \label{E:iperp1}
  \ffi^\perp \ = \ \op\, \ffi^\perp_k \quad\hbox{with}\quad
  \ffi^\perp_k \ = \ \bigoplus_{p+q=k} \, H_p^p \,\ot\, \overline{H_q^q} \,.
\end{equation}
Define $G_\bR$--homogeneous holomorphic vector bundles 
\begin{equation} \label{E:dfn_sH}
\renewcommand{\arraystretch}{1.5}
\begin{array}{rcl}
  \sH^\ell_m & \dfn & G_\bR \times_V H_m^\ell \,,\\
  \sH^\ell & \dfn & G_\bR \times_V H^\ell(\fg_-,\bC) \ = \ 
  \sH^\ell_\ell \,\op\,\sH^\ell_{\ell+1}\,\op\,\sH^\ell_{\ell+2}\,\op\cdots \,.
\end{array}
\end{equation}
By \eqref{E:Hconj}
$$
  \overline{\sH^\ell} \ \simeq \ G_\bR \times_V H^\ell(\fg_+,\bC) \,.
$$
Set
$$
  I^{\perp}_k \ = \ \displaystyle
  \bigoplus_{p+q = k} \sH^p_p \ot \overline{\sH^q_q} 
  \tand
  I^\perp \ = \ \displaystyle \bigoplus_k I^{\perp}_k \,,
$$
and let
\begin{equation} \label{E:dfn_cC}
  \cC \ \dfn \ \cC^\infty(I^\perp) \,,\quad
  \cC^k \ \dfn \ \cC^\infty(I^\perp_k) \tand
  \cC^{p,q} \ \dfn \ \cC^\infty(\sH^p_p\ot\overline{\sH^q_q})
\end{equation}
denote the smooth sections.   Equation \eqref{E:iperp1} yields \eqref{E:Cpq}.  The remainder of the Section \ref{S:outline} outline follows, and we have established

\begin{theorem} \label{T:cc}
The characteristic cohomology $H^\sb_\cI(D)$ of the infinitesimal period relation is the cohomology $H^\sb(\cC,\mathbf{d})$ of the total complex associated with the double complex $(\cC^{\sb,\sb},\d,\bar\d)$ of $G_\bR$--invariant differential operators.
\end{theorem}

\begin{remark}
Likewise, $H^\sb_\cI(U) = H^\sb(\cC_U,\bd)$ for any open subset $U \subset D$; however, the operators $\bd,\del,\bar\d$ are no longer $G_\bR$--invariant if $U \subsetneq D$.
\end{remark}

Define
\begin{equation} \label{E:dfn_s}
  \m \ \dfn \ \tmax \{ p \ | \ H^p_p \not= 0 \} \,.
\end{equation}
The double complex of Theorem \ref{T:cc} is as displayed in Figure \ref{f:d_cpx}.  The integer $\m$ is identified in the examples of Appendix \ref{S:egs}.
\begin{small}
\begin{figure}[h] 
\caption[The double complex]{The double complex of Theorem \ref{T:cc}.}
$$ 
\begin{array}{cccccccccc}
  0 & 
  & 0 &
  & 0 &
  & & & 0 & \\
  \uparrow\hbox{\scriptsize{$\bar\d$}} & 
  & \uparrow\hbox{\scriptsize{$\bar\d$}} &
  & \uparrow\hbox{\scriptsize{$\bar\d$}} &
  & & & \uparrow\hbox{\scriptsize{$\bar\d$}} & \\
  \cC^{0,\mu} & \stackrel{\d}{\longrightarrow} &
  \cC^{1,\mu} & \stackrel{\d}{\longrightarrow} & 
  \cC^{2,\mu} & \stackrel{\d}{\longrightarrow} & 
  \cdots & \stackrel{\d}{\longrightarrow} 
  & \cC^{\mu,\mu} & \stackrel{\d}{\longrightarrow} \ 0 \\
  \uparrow\hbox{\scriptsize{$\bar\d$}} & 
  & \uparrow\hbox{\scriptsize{$\bar\d$}} &
  & \uparrow\hbox{\scriptsize{$\bar\d$}} &
  & & & \uparrow\hbox{\scriptsize{$\bar\d$}} & \\
  \vdots & & \vdots & & \vdots & & & & \vdots & \\
  \uparrow\hbox{\scriptsize{$\bar\d$}} & 
  & \uparrow\hbox{\scriptsize{$\bar\d$}} &
  & \uparrow\hbox{\scriptsize{$\bar\d$}} &
  & & & \uparrow\hbox{\scriptsize{$\bar\d$}} & \\
  \cC^{0,2} & \stackrel{\d}{\longrightarrow} &
  \cC^{1,2} & \stackrel{\d}{\longrightarrow} & 
  \cC^{2,2} & \stackrel{\d}{\longrightarrow} & 
  \cdots & \stackrel{\d}{\longrightarrow} 
  & \cC^{\mu,2} & \stackrel{\d}{\longrightarrow} \ 0 \\
  \uparrow\hbox{\scriptsize{$\bar\d$}} & 
  & \uparrow\hbox{\scriptsize{$\bar\d$}} &
  & \uparrow\hbox{\scriptsize{$\bar\d$}} &
  & & & \uparrow\hbox{\scriptsize{$\bar\d$}} & \\
  \cC^{0,1} & \stackrel{\d}{\longrightarrow} &
  \cC^{1,1} & \stackrel{\d}{\longrightarrow} & 
  \cC^{2,1} & \stackrel{\d}{\longrightarrow} & 
  \cdots & \stackrel{\d}{\longrightarrow} 
  & \cC^{\mu,1} & \stackrel{\d}{\longrightarrow} \ 0 \\
  \uparrow\hbox{\scriptsize{$\bar\d$}} & 
  & \uparrow\hbox{\scriptsize{$\bar\d$}} &
  & \uparrow\hbox{\scriptsize{$\bar\d$}} &
  & & & \uparrow\hbox{\scriptsize{$\bar\d$}} & \\
  \cC^{0,0} & \stackrel{\d}{\longrightarrow} &
  \cC^{1,0} & \stackrel{\d}{\longrightarrow} & 
  \cC^{2,0} & \stackrel{\d}{\longrightarrow} & 
  \cdots & \stackrel{\d}{\longrightarrow} 
  & \cC^{\mu,0} & \stackrel{\d}{\longrightarrow} \ 0
\end{array}
$$
\label{f:d_cpx}
\end{figure}
\end{small}

\begin{remark}
By \cite[Theorem 3.12]{ MR3217458}, any variation of Hodge structure has dimension at most $\mu$.
\end{remark}

\section{Comparison of de Rham and characteristic cohomology} \label{S:cvc}

Define
\begin{equation} \label{E:dfn_lo}
  \nu \ \dfn \ \tmax\{ \ell \ | \ H^\ell_m = 0 \ \forall \ m > \ell \} \,.
\end{equation}
The main result of this section is Theorem \ref{T:hh} and its corollary \eqref{E:cvc} which establishes (i) the finite dimensionality of the characteristic cohomology in degree $k < \n$ (Corollary \ref{C:FD}), and (ii) a local Poincar\'e lemma for the characteristic cohomology differential (Corollary \ref{C:PL3}).

By \eqref{E:H1}
\begin{equation} \nonumber
  \nu \ > \ 0 \,,
\end{equation}
and \eqref{E:Hev} and \eqref{E:dfn_sH} yield
\[
  \sH^\ell \ = \ \sH^\ell_\ell \quad \hbox{for all } \ \ell \le \nu \,.
\]
The value $\nu$ is determined in the examples of Appendix \ref{S:egs}.

By \eqref{E:dfn_cC}, $\cC^{p,0}$ is the space of smooth sections of $\sH^p_p$.  Note that the differential $\d$ preserves holomorphic sections, yielding a complex
\begin{equation}\label{E:cpxll}
  0 \ \to \ \cO(\sH^0_0) \ \stackrel{\d}{\to} \
  \cO(\sH^1_1) \ \stackrel{\d}{\to} \
  \cO(\sH^2_2) \ \stackrel{\d}{\to} \cdots \stackrel{\d}{\to} \
  \cO(\sH^s_s) \ \to \ 0\,.
\end{equation} 
Given an open subset $U \subset D$, let $\bH^\sb(U,\sH^*_*)$ denote the hypercohomology of the complex \eqref{E:cpxll}.  (See \cite[\S3.5]{MR1288523} for a discussion of hypercohomology.)  

\begin{theorem} \label{T:hh}
Let $U \subset D$ be an open set.
{\bf (a)}
There exist identifications
$$ 
   H^k(U,\bC) \ = \  \bH^k(U,\sH^\ast_\ast) \quad \hbox{for all } \ k < \nu \,. 
$$
{\bf (b)}
There exists an inclusion
$$ 
  H^{\nu}(U,\bC) \ \inj \ \bH^{\nu}(U,\sH^\ast_\ast)\,.
$$
The cokernel of the inclusion admits an identification 
$$
  \bH^{\nu}(U,\sH^\ast_\ast)/H^{\nu}(U,\bC) \ = \ 
  \tker\{ d^\ddagger_{\nu+1} : 
  H^0(U,\cH^{\nu}) \to H^{\nu+1}(U,\bC) \} \,.
$$

\smallskip

\noindent{\bf (c)}
There exist filtrations ${}^\dagger{}F^\sb \bH^\sb(U,\sH^\ast_\ast)$ and $F^\sb H^\sb_\cI(U)$ of the hypercohomology and characteristic cohomology, respectively, such that the associated graded decompositions satisfy the following.  There exist identifications 
$$  
  {}^\dagger\tGr^\sb\bH^k(U,\sH^\ast_\ast) \ = \ 
  \tGr^\sb{}H^k_\cI(U) \quad \hbox{for all } \ k < \nu \,.
$$
For $k=\nu$ we have 
\[ 
\renewcommand{\arraystretch}{1.3}
\begin{array}{rcl}
  {}^\dagger\tGr^p\bH^{\nu}(U,\sH^\ast_\ast) & = & 
  \tGr^pH^{\nu}_\cI(U) \quad \hbox{for all } \ p \not=0 \,,\\
  {}^\dagger\tGr^0\bH^{\nu}(U,\sH^\ast_\ast) & \inj &
  \tGr^0{}H^{\nu}_\cI(U)\,.
\end{array}
\]
\smallskip

\noindent{\bf (d)}  In the case that $U = D$, each of the identifications, inclusions and filtrations above are as $G_\bR$--modules, and the map $d^\ddagger_{\nu+1}$ is $G_\bR$--equivariant.
\end{theorem}

\noindent The theorem is proved in Section \ref{S:prfhh}.  A discussion of the inclusion ${}^\dagger\tGr^0\bH^{\nu}(U,\sH^\ast_\ast) \inj \tGr^0{}H^{\nu}_\cI(U)$ in Theorem \ref{T:hh}(c) is given in Remark \ref{R:hh}. Together (a) and (c) of Theorem \ref{T:hh} yield (graded) identifications 
\begin{equation} \label{E:cvc}
  H^k(U,\bC) \ \simeq \ H^k_\cI(U) \quad \hbox{for} \quad k < \nu\,.
\end{equation}
This implies two corollaries.  First, 

\begin{corollary}[Finite--dimensionality] \label{C:FD}
The characteristic cohomology $H^k_\cI(D)$ is finite--dimensional for $k < \nu$, and zero when $k < \nu$ is odd.
\end{corollary}

\begin{proof}
This follows from the identification \eqref{E:cvc} and \cite[Proposition 4.3.5]{MR2188135}.
\end{proof}

\noindent Second, from \eqref{E:cvc} and the local exactness of the de Rham complex we obtain

\begin{corollary}[$\mathbf{d}$--Poincar\'e lemma] \label{C:PL3}
The operator $\mathbf{d} : \cC^k \to \cC^{k+1}$ is locally exact for $0 < k < \nu$.  That is, if $\phi \in \cC^k$ is $\mathbf{d}$--closed, then locally there exists $\psi \in \cC^{k-1}$ such that $\mathbf{d}\psi = \phi$.
\end{corollary}

\begin{remark}[Relationship to the Bryant--Griffiths characteristic cohomology] \label{R:BG}
Equation \eqref{E:cvc} and Corollary \ref{C:PL3} are very like results of Bryant and Griffiths on the (prolonged) characteristic cohomology of an involutive exterior differential system, \cf Theorem 1 of \S6.1 and Theorem 2 of \S4.2 in \cite{MR1311820}, respectively.  Given this similarity, it is natural to ask: what is the relationship between our $\nu$ and their $n-\ell$?  I've chosen not to investigate the question here, but would like to observe that these integers agree when the IPR is a contact distribution, \cf Section \ref{S:eg_adj} of this paper and Example 1 of \cite[\S6.3]{MR1311820}
\end{remark}

The following Theorems \ref{T:E2} and \ref{T:E1} will be used in the proof of Theorem \ref{T:hh}.  Given an open subset $U \subset D$, let
\begin{equation} \label{E:cpxcoh}
  H^p(\sH^\ast_\ast(U),\d) \ \dfn \ 
  \frac{\tker\{\d:\cO_U(\sH^p_p) \to \cO_U(\sH^{p+1}_{p+1}) \}}
       {\tim\{\d : \cO_U(\sH^{p-1}_{p-1}) \to \cO_U(\sH^{p}_{p})\}}
\end{equation}
denote the cohomology of the complex \eqref{E:cpxll} on $U$.

\begin{theorem} \label{T:E2}
Let $U \subset D$ be an open subset.
{\bf (a)}
There exist identifications
$$
  H^p(U,\bC) \ = \ H^p(\sH^\ast_\ast(U),\d) \quad\hbox{for all } \ p < \nu\,.
$$
{\bf (b)} There exists an inclusion 
$$
  H^{\nu}(U,\bC) \ \inj \  H^{\nu}(\sH^\ast_\ast(U),\d) \,.
$$
The image is $\bigcap_{i=2}^\infty \tker\,\del_i$, where 
\begin{eqnarray*}
  \del_2 : H^{\nu}(\sH^\ast_\ast(U),\d) & \to & 
  \tker\{ \del_1 : \cO_U(\sH^{\nu+1}_{\nu+2}) 
          \to \cO_U( \sH^{\nu+2}_{\nu+3}) \}\\
  \del_{i+1} : \tker\,\del_i & \to & 
  \tker\{ \del_1 : \cO_U(\sH^{\nu+1}_{\nu+i}) 
          \to \cO_U( \sH^{\nu+2}_{\nu+i+1}) \} \,,\quad i \ge 2 \,.
\end{eqnarray*}
\noindent{\bf (c)} When $U = D$, the identifications and inclusions above are as $G_\bR$--modules, and the maps $\del_i$ are $G_\bR$--equivariant.
\end{theorem}
 
\noindent The theorem is proved in Section \ref{S:prfE2}.  Theorem \ref{T:E2}(a) and the local exactness of the complex $(\Omega^\sb, \del)$ yield a holomorphic Poincar\'e lemma for the operators $\d : \cO(\sH^p_p) \to \cO(\sH^{p+1}_{p+1})$.

\begin{corollary}[Holomorphic $\d$--Poincar\'e lemma] \label{C:PL2}
The operator $\d : \cO(\sH^p_p) \to \cO(\sH^{p+1}_{p+1})$ is locally exact for $0 < p < \nu$.  That is, if $\phi \in \cO(\sH^p_p)$ is $\d$--closed, then locally there exists $\psi \in \cO(\sH^{p-1}_{p-1})$ such that $\d\psi = \phi$.
\end{corollary}

Let $H^q(U,\sH^p_p)$ denote the cohomology of the sheaf of holomorphic sections of $\sH^p_p$.  

\begin{theorem} \label{T:E1}
Let $U \subset D$ be an open subset.
{\bf (a)}
There exist identifications 
$$
  H^q(U,\sH^p_p) \ = \ H^q( \cC^{p,\sb}_U , \bar\d ) \quad\hbox{for all } \ 
  q < \nu\,.
$$
{\bf (b)}
There exists an inclusion 
$$
  H^{\nu}(U,\sH^p_p) \ \inj \ H^{\nu}(\cC^{p,\sb}_U , \bar\d) \,.
$$
The image is $\bigcap_{i=2}^\infty \tker\,\bar\del_i$, where 
\begin{eqnarray*}
  \bar\del_2 : H^{\nu}(\cC^{p,\sb}_U , \bar\d) & \to & 
  \tker\{ \bar\del_1 : 
           \cC^\infty_U(\sH^p_p\ot\overline{\sH^{\nu+1}_{\nu+2}}) 
       \to \cC^\infty_U(\sH^p_p\ot\overline{\sH^{\nu+2}_{\nu+3}}) \}\\
  \bar\del_{i+1} : \tker\,\bar\del_i & \to & 
  \tker\{ \bar\del_1 : 
          \cC^\infty_U(\sH^p_p\ot\overline{\sH^{\nu+1}_{\nu+i}}) 
      \to \cC^\infty_U(\sH^p_p\ot\overline{\sH^{\nu+2}_{\ell+i+1}}) \} 
  \,,\quad i \ge 2 \,,
\end{eqnarray*}
\noindent{\bf (c)} When $U = D$, the identifications and inclusions above are as $G_\bR$--modules, and the maps $\bar\del_i$ are $G_\bR$--equivariant.
\end{theorem}

\noindent The theorem is proved in Section \ref{S:prfE1}.  Theorem \ref{T:E1}(a) and the local exactness of the Dolbeault resolution of $\sH^p_p$ yield a Poincar\'e lemma for the operators $\bar\d$.

\begin{corollary}[$\bar\d$--Poincar\'e lemma] \label{C:PL1a}
The operator $\bar\d : \cC^{\sb,q} \to \cC^{\sb,q+1}$ is locally exact for $0 < q < \nu$.  That is, if $\phi \in \cC^{\sb,q}$ is $\bar\d$--closed, then locally there exists $\psi \in \cC^{\sb,q-1}$ such that $\bar\d\psi = \phi$.
\end{corollary}

\noindent Taking conjugates we obtain

\begin{corollary}[$\d$--Poincar\'e lemma] \label{C:PL1b}
The operator $\d : \cC^{p,\sb} \to \cC^{p+1,\sb}$ is locally exact for $0 < p < \nu$.  That is, if $\phi \in \cC^{p,\sb}$ is $\d$--closed, then locally there exists $\psi \in \cC^{p-1,\sb}$ such that $\d\psi = \phi$.
\end{corollary}

To emphasize the $G_\bR$--module structure we will prove the results of Section \ref{S:cvc} for 
\[
  U \ = \ D \,.
\]
The results for arbitrary open sets $U \subset D$ follow by identical arguments.

\subsection{Weighted filtration of forms} \label{S:filt} 

The basic idea underlying the proofs of Theorems \ref{T:E2} and \ref{T:E1} is presented in this section.  The spectral sequences that arise are induced by filtrations that are variants of the basic filtration \eqref{E:filtp0} introduced here.  For each of these variants we will have analogs of Lemma \ref{L:circE} and Corollary \ref{C:circE}, and the theorems are essentially these analogs.

Recall the definition \eqref{E:dfn_cT}.  Define a splitting
$$
   \cT D \ = \ \ \bigoplus_{\ell>0} \, \cT_\ell 
  \quad\hbox{by}\quad \cT_\ell \ = \ G \times_V \fg_{-\ell} \,,
$$
and a filtration
$$
  F_\ell( \cT D) \ = \ \cT_1 \,\op\,\cT_2 
  \,\op\cdots\op\,\cT_\ell \,.
$$
The relation \eqref{E:grLB} yields 
\begin{equation} \label{E:br}
  [ F_a(\cT D) , F_b(\cT D)] \ \subset \ 
   F_{a+b}(\cT D) \,.
\end{equation}

Recall the definition \eqref{E:dfn_pq} and equation \eqref{E:cT*}.  Define a splitting of $ \cT D^{\,*}$ by
\begin{subequations} \label{SE:split*}
\begin{equation}
  \cT D^{\,*} \ = \ \tw^{1,0}_D \ = \ \bigoplus_{\ell} \tw^{1,0}_\ell \,,
\end{equation}
where
\begin{equation}
  \tw^{1,0}_\ell \ \dfn \ G \times_V \fg_{-\ell}^* \ \simeq \ G \times_V \fg_{\ell} \,,
\end{equation}
\end{subequations}
and a filtration on $\tw^{p,0}_D$ by 
\begin{equation} \label{E:filtp0}
  F^\ell( \tw^{p,0}_D ) \ \dfn \ 
  \tim\left\{  \bigoplus_{\sum b_i \ge \ell}
  \tw^{1,0}_{b_1} \ot \cdots \ot \tw^{1,0}_{b_p}
  \ \to \ \tw^{p,0}_D
  \right\} \,.
\end{equation}
For example,
$$
  F^\ell(\tw^{1,0}_D) \ = \ \tw^{1,0}_\ell \,\op\,\tw^{1,0}_{\ell+1} \,\op\,
  \tw^{1,0}_{\ell+2}\,\op\cdots
$$
is the annihilator of $F^{\ell-1}( \cT D )$ in $\cT D^{\,*} = \tw^{1,0}_D$.

The filtration \eqref{E:filtp0} induces a filtration $F^\sb( \cA^{p,0})$ on the smooth $(p,0)$--forms.  Moreover, \eqref{E:br} implies $\del$ preserves the filtration
\begin{equation} \label{E:d'}
  \del\,F^\ell(\cA^{\sb,0}) \ \subset \ F^\ell( \cA^{\sb,0}) \,.
\end{equation}
Thus we obtain a spectral sequence $\{ \del_i : {}^\circ{}E_i^{\ell,-m} \to {}^\circ{}E_i^{\ell+i,1-m-i} \}$ abutting to the cohomology of the complex $(\cA^{\sb,0},\del)$,
$$
  {}^\circ{}E_i \ \Longrightarrow \ H(\cA^{\sb,0},\del) \,.
$$
Note that $F^\ell\cA^{p,0} = \cA^{p,0}$ if $\ell \le p$, so that the associated graded is
$$
  {}^\circ{}E_0^{\ell,-m} \ = \ \frac{F^\ell \cA^{\ell-m,0}}{F^{\ell+1}\cA^{\ell-m,0}}\,,
$$
and the spectral sequence `lives' in the lower--right quadrant, \cf Figure \ref{f:circE0}.
\begin{small}
\begin{figure}[h] 
\caption[The page ${}^\circ{}E_0^{\ell,-m}$]{The page  ${}^\circ{}E_0^{\ell,-m} = \cA_\ell^{\ell-m,0}$.}
$$
\renewcommand{\arraystretch}{1.3}
\begin{array}{cccccc}
 \cA^{0} & \cA^{1,0}_1 & \cA^{2,0}_2 & \cA^{3,0}_3 & \cA^{4,0}_4 & \cdots \\
  & & \uparrow\hbox{\scriptsize{$\del_0$}} 
    & \uparrow\hbox{\scriptsize{$\del_0$}} 
    & \uparrow\hbox{\scriptsize{$\del_0$}} &  \\ 
 0 & 0 & \cA^{1,0}_2 & \cA^{2,0}_3 & \cA^{3,0}_4 & \cdots \\
   & & & \uparrow\hbox{\scriptsize{$\del_0$}} 
   & \uparrow\hbox{\scriptsize{$\del_0$}} &  \\ 
 0 & 0 & 0 & \cA^{1,0}_3 & \cA^{2,0}_4 & \cdots \\
   & & & & \uparrow\hbox{\scriptsize{$\del_0$}} &  \\ 
 0 & 0 & 0 & 0 & \cA^{1,0}_4 & \cdots \\
 \vdots & \vdots & \vdots & \vdots & \vdots & \\
\end{array}
$$
\label{f:circE0}
\end{figure}
\end{small}

Let 
$$
  \cA^{p,0}_\ell \ \simeq \ F^\ell(\cA^{p,0})/F^{\ell+1}(\cA^{p,0}) 
  \ = \ {}^\circ{}E^{\ell,p-\ell}_0
$$
denote the smooth sections of 
\begin{equation} \label{E:twgrad}
  \frac{F^\ell( \tw^{p,0}_D )}{F^{\ell+1}( \tw^{p,0}_D )} \ \simeq \ 
  \tim\left\{  \bigoplus_{\sum b_i = \ell}
  \tw^{1,0}_{b_1} \ot \cdots \ot \tw^{1,0}_{b_p}
  \ \to \ \tw^{p,0}_D
  \right\} \ \dfn \ \tw^{p,0}_\ell \,.
\end{equation}
It will be helpful to note that $\tw^{p,0}_\ell$ admits the following description as a $G_\bR$--homogeneous vector bundle.  Let 
$$
  \tw^p \fg_-^* \ = \ \tw^p_{p}\,\fg_-^* \ \op\ \tw^p_{p+1}\,\fg_-^* \ \op\ 
  \tw^p_{p+2}\,\fg_-^* \ \op\cdots
$$
be the $\ttE$--eigenspace decomposition of $\tw^p\fg_-^*$; here $\ttE$ acts on $\tw^p_{\ell}\,\fg_-^*$ by the scalar $\ell$.  Then 
$$
  \tw^{p,0}_\ell \ = \ G_\bR \times_V \tw^p_{\ell}\,\fg_-^* \,.
$$
Given $\phi \in \tw^p_\ell\fg_-^*$ and $X_i \in \fg_{-a_i}$, with $0<a_i$, observe that 
\begin{equation} \label{E:tw_ell}
  \phi(X_1,\ldots,X_p) \ \not= \ 0 \ \hbox{ only if } \ 
  {\textstyle{\sum}} a_i \,=\, \ell \,.
\end{equation}

\begin{lemma} \label{L:circE}
The $G_\bR$--module ${}^\circ{}E^{\ell,-m}_1$ is naturally identified with the smooth sections of 
$\sH^{\ell-m}_{\ell}$.  Moreover, $( {}^\circ{}E^{\sb,0}_1 , \del_1 ) = ( \cC^{\sb,0} , \d )$, so that ${}^\circ{}E^{p,0}_2 = H^p( \cC^{\sb,0} , \d )$ as $G_\bR$--modules.
\end{lemma}

\begin{proof}
We will show that the vertical differential $\del_0$ is algebraic; in fact, it is given (up to a sign) by the Lie algebra cohomology differential $\e : \tw^p \fg_-^* \to \tw^{p+1} \fg_-^*$ of Section \ref{S:lac}.  This is seen as follows.  Let $\w$ denote the $\fg_\bC$--valued left-invariant Maurer-Cartan form on $G_\bR$, and let $\w_-$ denote the $\fg_-$--valued component.  Given a local section $D \to G_\bR$, we abuse notation and let $\w$ and $\w_-$ also denote the pull-backs to $D$.  Locally, any $\phi \in \cA^{p,0}$ is of the form $\phi = f(\w_-\wedge\cdots\wedge\w_-)$ where $f : D \to \tw^p\fg_-^*$ a smooth, locally defined function.  Likewise, any $\phi \in \cA^{p,0}_\ell$  is of the form $\phi = g(\w_-\wedge\cdots\wedge\w_-)$ with $g : D \to \tw^p_{\ell}\,\fg_-^*$ is a smooth, locally defined function.  (To be precise, we regard $g$ as a map to $\tw^q\fg_-^*$ taking values in the annihilator of $\op_{m\not=\ell}\, \tw^q_{m}\fg_-$.)

Fix $\phi \in \cA^{p,0}_\ell = {}^\circ{}E^{\ell,p-\ell}_0$.  From \eqref{E:tw_ell} we see that to compute the differential $\del_0 \phi \in \cA^{p+1,0}_{\ell}$ it suffices to compute $(\del_0 \phi)(\xi_0 , \xi_1 , \ldots, \xi_p)$ where $\xi_i$ is a smooth section of $\cT_{a_i}$ and $\sum a_i = \ell$.  Without loss of generality, we may assume that $\w_-(\xi_i) = X_i \in \fg_{-a_i}$ is constant.  Then
\begin{eqnarray*}
  (\del_0 \phi)(\xi_0 , \xi_1 , \ldots, \xi_p) & = & 
  \sum_i (-1)^i \xi_i \,\phi(\xi_0 , \ldots , \hat \xi_i , \ldots , \xi_p) \\
  & & - 
  \sum_{i<j} (-1)^{i+j} \phi\left( [ \xi_i , \xi_j ] , 
     \xi_0 , \ldots , \hat \xi_i , \ldots , \hat \xi_j , \ldots , \xi_p\right) \,.
\end{eqnarray*}
By \eqref{E:tw_ell}, we have $\phi(\xi_0 , \ldots , \hat \xi_i , \ldots , \xi_p) = 0$.
Therefore, 
\begin{eqnarray*}
  (\del_0 \phi)(\xi_0 , \xi_1 , \ldots, \xi_p) & = & 
  -\sum_{i<j} (-1)^{i+j} \phi\left( [ \xi_i , \xi_j ] , \xi_0 , \ldots , 
  \hat \xi_i , \ldots , \hat \xi_j , \ldots , \xi_p\right) \\
  & = & -\sum_{i<j} (-1)^{i+j} f\left( [ X_i , X_j ] , 
     X_0 , \ldots , \hat X_i , \ldots , \hat X_j , \ldots , X_p\right) \\
  & = & -(\e f)(X_0 , \ldots , X_p) \,.
\end{eqnarray*}
Therefore, the differential $\del_0 : \cA^{p,0}_\ell \to \cA^{p+1}_{\ell}$ is the map naturally induced by restriction of $-\e : \tw^p\fg_-^* \to \tw^{p+1}\fg^*_-$ to the $\ttE$--eigenspace $\tw^p_{\ell}\,\fg_-^*$ of eigenvalue $\ell$.  It now follows from \eqref{E:Hev} and \eqref{E:dfn_sH} that ${}^\circ{}E^{\ell,-m}_1 = \cC^\infty(\sH^{\ell-m}_{\ell})$, establishing the first half of the lemma.  From the definition \eqref{E:dfn_cC}, we see that ${}^\circ{}E^{\ell,0}_1 = \cC^{\ell,0}$.  The final assertion that $\del_1 = \d$ is straightforward definition chasing.
\end{proof}

\begin{corollary} \label{C:circE}
{\bf (a)}
There exist $G_\bR$--module identifications 
$$
  H^p(\cA^{\sb,0},\del) \ = \ H^p(\cC^{\sb,0},\d) 
  \quad\hbox{for all } \ p < \nu \,.
$$  
{\bf (b)}
There exists a $G_\bR$--module inclusion 
$$
  H^{\nu}(\cA^{\sb,0},\del) \ \inj \ H^{\nu}(\cC^{\sb,0},\d) \,.
$$  
The image is $\bigcap_{i=2}^\infty \tker\,\del_i$, where 
\begin{eqnarray*}
  \del_2 : H^{\nu}(\cC^{\sb,0} , \d) & \to & 
  \tker\{ \del_1 : \cC^\infty(\sH^{\nu+1}_{\nu+2}) 
          \to \cC^\infty( \sH^{\nu+2}_{\nu+3}) \}\\
  \del_{i+1} : \tker\,\del_i & \to & 
  \tker\{ \del_1 : \cC^\infty(\sH^{\nu+1}_{\nu+i}) 
          \to \cC^\infty( \sH^{\nu+2}_{\nu+i+1}) \} \,,\quad i \ge 2 \,.
\end{eqnarray*}
and each $\del_i$ is a $G_\bR$--equivariant map. 
\end{corollary}

\begin{proof}
Recall the definitions \eqref{E:dfn_sH} and \eqref{E:dfn_lo}; the identification of ${}^\circ{}E^{\ell,-m}_1$ with $\cC^\infty(\sH^{\ell-m}_\ell)$ by Lemma \ref{L:circE} implies that 
\begin{equation} \nonumber
  {}^\circ{}E^{\ell,-m}_1 \ = \ 0 \,, \quad\hbox{for all } \  
  m > 0 \ \hbox{ and } \ \ell-m \le \nu \,,
\end{equation}
\cf Figure \ref{f:circE1}.
\begin{small}
\begin{figure}[h] 
\caption{The page ${}^\circ{}E_1^{\ell,-m} = \cC^\infty( \sH^{\ell-m}_\ell)$.}
$$
\renewcommand{\arraystretch}{1.5}
\begin{array}{cccccccc}
  {}^\circ{}E_1^{0,0} & \cdots & {}^\circ{}E_1^{\nu,0} & {}^\circ{}E_1^{\nu+1,0} 
  & {}^\circ{}E_1^{\nu+2,0} & {}^\circ{}E_1^{\nu+3,0} & {}^\circ{}E_1^{\nu+4,0} 
  & \cdots \\
  0 & \cdots & 0 & 0 & {}^\circ{}E_1^{\nu+2,-1} & {}^\circ{}E_1^{\nu+3,-1} 
  & {}^\circ{}E_1^{\nu+4,-1} & \cdots \\
  0 & \cdots & 0 & 0 & 0 & {}^\circ{}E_1^{\nu+3,-2} 
  & {}^\circ{}E_1^{\nu+4,-2} & \cdots \\
  0 & \cdots & 0 & 0 & 0 & 0 & {}^\circ{}E_1^{\nu+4,-3} & \cdots \\
  \vdots & & \vdots & \vdots & \vdots & \vdots & \vdots & 
\end{array}
$$
\label{f:circE1}
\end{figure}
\end{small}
Since the spectral sequence abuts to the cohomology $H(\cA^{\sb,0},\del)$, we see that
$$
  {}^\circ{}E^{p,0}_\infty \ =\ H^p(\cA^{\sb,0},\del) 
  \quad\hbox{for all} \ p \le \nu \,.
$$
In the case that $p < \nu$, we have ${}^\circ{}E^{p,0}_\infty = {}^\circ{}E^{p,0}_2$.  This yields the first half of the corollary.

In the case that $p=\nu$, we see that 
$$
  {}^\circ{}E_{i+1}^{\nu,0} \ = \ 
  \tker\{ \del_i : {}^\circ{}E^{\nu,0}_i \to {}^\circ{}E^{\nu+i,1-i}_i \}
  \quad\hbox{for all } \ i \ge 2 \,.
$$
Thus,
$$
  {}^\circ{}E_\infty^{\nu,0} \ = \ 
  \bigcap_{i=2}^\infty \tker\{ \del_i : {}^\circ{}E^{\nu,0}_i 
  \to {}^\circ{}E^{\nu+i,1-i}_i \} \ \subset \ {}^\circ{}E_2^{\nu,0}\,,
$$
yielding the second half of the corollary.
\end{proof}

Before continuing to the proofs of the theorems, we briefly discuss the conjugate versions of the filtration \eqref{E:filtp0}, Lemma \ref{L:circE} and Corollary \ref{C:circE}.  By \eqref{E:conjgell} and \eqref{E:dfn_pq}, we have $\tw^{0,q}_D = \overline{\tw^{q,0}_D}$.  Given \eqref{E:filtp0}, we may define a filtration
\begin{equation} \label{E:filt0q}
  F^\ell( \tw^{0,q}_D ) \ \dfn \ \overline{F^\ell( \tw^{q,0}_D )}\,.
\end{equation}
Let $F^\ell(\cA^{0,\sb})$ denote the corresponding filtration of $\cA^{0,\sb}$.  Note that $F^\ell(\cA^{0,q}) = \overline{F^\ell(\cA^{q,0})}$.  And so, by \eqref{E:d'} the differential $\bar\del$ preserves the filtration
\begin{equation} \label{E:d''}
  \bar\del\,F^\ell(\cA^{0,\sb}) \ \subset \ F^\ell(\cA^{0,\sb}) \,.
\end{equation}
Since $(\cA^{0,\sb},\bar\del)$ is the Dolbeault resolution of $\cO$, we see that the filtration gives rise to a spectral sequence $\{ \bar\del_i : {}^\star{}E^{\ell,-m}_i \to {}^\star{}E^{\ell+i,1-m-i}_i \}$ abutting to the sheaf cohomology $H^\sb(D,\cO)$,
$$
  {}^\star{}E_i \ \Longrightarrow \ H(\cA^{0,\sb},\bar\del) 
  \,=\, H^\sb(D,\cO) \,.
$$

\begin{lemma} \label{L:starE}
The $G_\bR$--module ${}^\star{}E^{\ell,-m}_1$ is naturally identified with the smooth sections of 
$\overline{\sH^{\ell-m}_{\ell}}$.  Moreover, $({}^\star{}E^{\sb,0}_1,\bar\del_1) = (\cC^{0,\sb},\bar\d)$, so that $  {}^\star{}E^{q,0}_2 = H^q\left( \cC^{0,\sb} \,,\, \bar\d \right)$ as $G_\bR$--modules.
\end{lemma}

\begin{proof}
The proof is entirely analogous to that of Lemma \ref{L:circE}: again, the vertical differential $\bar\del_0$ is algebraic, and given (up to a sign) by the Lie algebra cohomology differential $\e : \tw^p \fg_+^* \to \tw^{p+1} \fg_+^*$.   Details are left to the reader.
\end{proof}

The identification of ${}^\star{}E^{\ell,-m}_1$ with $\cC^\infty(\overline{\sH^{\ell-m}_\ell})$ implies that the page ${}^\star{}E_1$ is also of the form depicted in Figure \ref{f:circE1}.  Whence we obtain the following analog of Corollary \ref{C:circE}.

\begin{corollary} \label{C:starE}
{\bf (a)}
There exist $G_\bR$--module identifications 
$$
  H^q(D,\cO) \ = \ H^q(\cC^{0,\sb},\bar\d) \quad\hbox{for all } \ q < \nu \,.
$$
{\bf (b)}
There exists a $G_\bR$--module inclusion 
$$
  H^{\nu}(D,\cO) \ \inj \ H^{\nu}(\cC^{0,\sb},\bar\d) \,.
$$
The image is $\bigcap_{i=2}^\infty \tker\,\bar\del_i$, where 
\begin{eqnarray*}
  \bar\del_2 : H^{\nu}(\cC^{0,\sb} , \bar\d) & \to & 
  \tker\{ \bar\del_1 : \cC^\infty(\overline{\sH^{\nu+1}_{\nu+2}}) 
          \to \cC^\infty( \overline{\sH^{\nu+2}_{\nu+3}}) \}\\
  \bar\del_{i+1} : \tker\,\bar\del_i & \to & 
  \tker\{ \bar\del_1 : \cC^\infty(\overline{\sH^{\nu+1}_{\nu+i}}) 
          \to \cC^\infty( \overline{\sH^{\nu+2}_{\nu+i+1}}) \} \,,\quad i \ge 2 \,,
\end{eqnarray*}
and each $\bar\del_i$ is a $G_\bR$--equivariant map.
\end{corollary}

\noindent Note that Corollary \ref{C:starE} yields Theorem \ref{T:E1} in the case that $p=0$.

\subsection{Proof of Theorem \ref{T:E2}} \label{S:prfE2}

Let $\Omega^p = \cO(\tw^{p,0}_D)$ denote the holomorphic $(p,0)$--forms, and note that the complex $(\Omega^\sb,\del)$ is a resolution of $\bC$.  The filtration \eqref{E:filtp0} induces a filtration $F^\sb( \Omega^p)$, and \eqref{E:br} implies $\del$ preserves the filtration $\del\,F^\ell(\Omega^\sb) \subset F^\ell( \Omega^\sb)$.  Thus we obtain a spectral sequence abutting to the sheaf cohomology $H^\sb(D,\bC)$.  Arguments identical to those establishing Lemma \ref{L:circE} and Corollary \ref{C:circE} yield the theorem.

\subsection{Proof of Theorem \ref{T:E1}} \label{S:prfE1}

Recall the definitions \eqref{E:dfn_pq} and \eqref{E:dfn_sH}.  Let $\cA^{0,q}(\sH^p_p)$ denote the smooth sections of $\sH^p_p\ot\tw^{0,q}$, and note that the complex $(\cA^{0,\sb}(\sH^p_p),\bar\del)$ is the Dolbeault resolution of the holomorphic sections $\cO(\sH^p_p)$.  Recall the filtration \eqref{E:filt0q}, and define $F^\ell(\sH^p_p\ot\tw^{0,q}_D) \dfn \sH^p_p\ot F^\ell(\tw^{0,q}_D)$.  Let $F^\ell(\cA^{0,q}(\sH^p_p))$ denote the corresponding filtration of the smooth sections.  By \eqref{E:d''} the differential $\bar\del$ preserves the filtration $F^\ell(\cA^{0,\sb}(\sH^p_p))$.  Whence we obtain a spectral sequence $\{ \bar\del_i : {}^\star{}E^{\ell,-m}_i \to {}^\star{}E^{\ell+i,1-m-i}_i \}$ abutting to sheaf cohomology $H^\sb(D,\sH^p_p)$.  Keeping in mind that $\cC^{p,q}$ is the space of smooth sections of $\sH^p_p\ot \overline{\sH^q_q}$, \cf \eqref{E:dfn_cC}, an argument identical to that establishing Lemma \ref{L:starE} and Corollary \ref{C:starE} yields the theorem.  Details are left to the reader.

\begin{remark} \label{R:BGGres1}
It is sometimes the case that a simple argument with the spectral sequence $\{ {}^\star{}E_i,\bar\del_i\}$ yields a significant strengthening of Theorem \ref{T:E1}: under suitable conditions on the set $\{ (p,\ell) \ | \ H^p_\ell \not=0 \}$ there exist differential operators $\nabla : \cC^\infty(\sH^p_p \ot \overline{\sH^q})  \to \cC^\infty(\sH^p_p \ot \overline{\sH^{q+1}})$ with the properties that 
$$
  \nabla \ = \ \bar\d \quad\hbox{for } \ q < \nu \,,
$$
and
\begin{equation} \label{E:BGGres}
\renewcommand{\arraystretch}{1.5}
\begin{array}{rcl}
  0 \ \to \ \cO(\sH^p_p) & \inj & 
  \cC^\infty( \sH^p_p \ot\overline{\sH^0} ) \ \stackrel{\nabla}{\longrightarrow} \ 
  \cC^\infty( \sH^p_p \ot\overline{\sH^1} )\ \stackrel{\nabla}{\longrightarrow} \\
  & & \hbox{\hspace{56pt}}
  \cdots \stackrel{\nabla}{\longrightarrow} \ 
  \cC^\infty( \sH^p_p \ot\overline{\sH^d} ) \ \longrightarrow \ 0
\end{array}
\end{equation}
is a resolution of the sheaf $\cO(\sH^p_p)$ of holomorphic sections of $\sH^p_p$.  (Note that the definition \eqref{E:dfn_lo} implies $\sH^q = \sH^q_q$ for all $q < \nu$.)  For example, the resolution \eqref{E:BGGres} exists when $\cT_1$ is a contact distribution (equivalently, $\check D$ is an adjoint variety).  This and other examples are discussed in Appendix \ref{S:egs}.

It is interesting to compare the resolution \eqref{E:BGGres} with the Dolbeault resolution $(\cA^{0,\sb}(\sH^p_p),\bar\del)$.  Both resolutions have the same length.  The advantage of \eqref{E:BGGres} is that the vector bundles involved have smaller rank; that is, $\trank\,\overline{\sH^q} \le \trank\,\tw^{0,q}_D$, and this inequality is strict if and only if the containment $T_1 \subset TD$ is strict.  However, the price we pay for this reduction is that the operators $\nabla$ will generally not be of first-order.

The resolution \eqref{E:BGGres} may be viewed as a Dolbeault analog of the (generalized) Bernstein-Gelfand-Gelfand resolution of $\bC$ by differential operators on $\check D$, \cf \cite{MR1038279, MR0578996, MR0476813, MR586721}.\end{remark}

\subsection{Proof of Theorem \ref{T:hh}} \label{S:prfhh}

As we will see, the theorem follows from Corollary \ref{C:PL2} and Theorem \ref{T:E1} via standard spectral sequence arguments.

\subsubsection*{A spectral sequence for the characteristic cohomology}

Associated to the double complex $(\cC,\d,\bar\d)$ are standard filtrations of $\cC^\sb$, one of which is 
$$
  F^p\cC^{p+q} \ \dfn \ \bigoplus_{i\ge0} \cC^{p+i,q-i} \,.
$$
It is straightforward to confirm that $\bd$ preserves $F^p\cC^\sb$.  Whence the filtration induces a spectral sequence $\{ \bd_i : E^{p,q}_i \to E^{p+i,q+1-i}_i \}$ abutting to the characteristic cohomology
$$
  E_i \ \Longrightarrow \ H^\sb(\cC,\bd) \,=\, H^\sb_\cI(D) \,.
$$
As is well known 
\begin{eqnarray}
  \nonumber
  E^{p,q}_0 & = & \cC^{p,q} \quad \hbox{with }  \bd_0 \,=\, \bar\d\,,\\
  \label{E:ssIE}
  E^{p,q}_1 & = & H^q(\cC^{p,\sb} , \bar\d) 
  \quad \hbox{with }  \bd_1 \,=\, \d \,,\\
  \nonumber
  E^{p,q}_2 & = &  H^p( H^q(\cC^{\sb,\sb} , \bar\d) \,,\, \d )  \,.
\end{eqnarray}
From \eqref{E:ssIE} and Theorem \ref{T:E1} we see that 
\begin{equation} \label{E:E1<ro}
   E^{p,q}_1 \ = \ H^q(D,\sH^p_p) \quad \hbox{for all } \ q < \nu \,.
\end{equation}
Visually, up to the $q=\nu-1$ level, the $E_1$--page is given by sheaf cohomology, \cf Figure \ref{f:E1}.
\begin{small}
\begin{figure}[h] 
\caption[The $E_1$--page]{The $E_1$--page.}
\[ \renewcommand{\arraystretch}{1.5}
\begin{array}{ccccccc}
  H^s(\cC^{0,\sb},\bar\d) & \stackrel{\d}{\to} &
  H^s(\cC^{1,\sb},\bar\d) & \stackrel{\d}{\to} & 
  \cdots & \stackrel{\d}{\to} 
  & H^s(\cC^{s,\sb},\bar\d) \\
  \vdots & & \vdots & & & & \vdots \\
  H^{\nu}(\cC^{0,\sb},\bar\d) & \stackrel{\d}{\to} &
  H^{\nu}(\cC^{1,\sb},\bar\d) & \stackrel{\d}{\to} & 
  \cdots & \stackrel{\d}{\to} 
  & H^{\nu}(\cC^{s,\sb},\bar\d)  \\
  H^{\nu-1}(D,\sH^0_0) & \stackrel{\d}{\to} &
  H^{\nu-1}(D,\sH^1_1) & \stackrel{\d}{\to} & 
  \cdots & \stackrel{\d}{\to} 
  & H^{\nu-1}(D,\sH^s_s) \\
  \vdots & & \vdots & & & & \vdots \\
  H^{1}(D,\sH^0_0) & \stackrel{\d}{\to} &
  H^{1}(D,\sH^1_1) & \stackrel{\d}{\to} & 
  \cdots & \stackrel{\d}{\to} 
  & H^{1}(D,\sH^s_s) \\
  H^0(D,\sH^0_0) & \stackrel{\d}{\to} &
  H^0(D,\sH^1_1) & \stackrel{\d}{\to} & 
  \cdots & \stackrel{\d}{\to} 
  & H^0(D,\sH^s_s) \\
\end{array}
\]
\label{f:E1}
\end{figure}
\end{small}
Keeping \eqref{E:ssIE} in mind and consulting Figure \ref{f:E1}, we see that
\begin{equation} \label{E:E2}
  E^{p,q}_2 \ = \ H^p( H^q(D,\sH^\ast_\ast) , \d ) 
  \quad \hbox{for all } \ q < \nu \,.
\end{equation}

\subsubsection*{Two spectral sequences for the hypercohomology}

Let $\cH^\sb$ denote the cohomology sheaves of \eqref{E:cpxll}.  The two spectral sequences $\{ d^\dagger_i : {}^\dagger E^{p,q}_i \to {}^\dagger E^{p+i,q-i+1}_i \}$ and $\{ d^\ddagger_i : {}^\ddagger E^{p,q}_i \to {}^\ddagger E^{p-i+1,q+i}_i \}$  associated with the hypercohomology satisfy
\begin{equation}\label{E:dE2}
  {}^\dagger E^{p,q}_2 \ = \ H^p( H^q(D,\sH^\ast_\ast)\,,\, \d ) \tand
  {}^\ddagger E^{p,q}_2 \ = \ H^q( D , \cH^p ) \,.
\end{equation}

\subsubsection*{Proof of Theorem \ref{T:hh}(c)}

Equations \eqref{E:E2} and \eqref{E:dE2} yield
\begin{subequations} \label{SE:dEvE}
\begin{equation}
  {}^\dagger{}E_2^{p,q} \ = \ E_2^{p,q} \quad\hbox{for all } q < \nu \,.
\end{equation}
Moreover, \eqref{E:dE2}, Theorem \ref{T:E1} and \eqref{E:ssIE} yield  
\begin{eqnarray}
  \nonumber
  {}^\dagger{}E_2^{0,\nu} \, = \, H^0( H^{\nu}(D,\sH^\ast_\ast) , \d ) & = &  
  \tker\{ \d : H^{\nu}(D,\sH^0_0) \to H^{\nu}(D,\sH^1_1) \} \\
  & \subset & 
  \tker\{ \d : H^{\nu}(\cC^{0,\sb},\bar\d) 
  \to H^{\nu}(\cC^{1,\sb},\bar\d) \} \\
  \nonumber
  & = & H^0( H^{\nu}(\cC^{\sb,\sb} , \bar\d) \,,\, \d ) \ = \ E_2^{0,\nu}\,.
\end{eqnarray}
\end{subequations}
Visually, the inclusions ${}^\dagger{}E_2^{p,q} \subseteq E_2^{p,q}$ of \eqref{SE:dEvE} are depicted in Figure \ref{f:dEvE}.
\begin{small}
\begin{figure}[!h] 
\caption{The inclusions ${}^\dagger{}E_2^{p,q} \subseteq E_2^{p,q}$.}
$$ \renewcommand{\arraystretch}{1.5}
\begin{array}{cccc}
{}^\dagger{}E_2^{0,\nu} \subset E_2^{0,\nu} & \ast & \ast \\
{}^\dagger{}E_2^{0,\nu-1} = E_2^{0,\nu-1} 
  & {}^\dagger{}E_2^{1,\nu-1} = E_2^{1,\nu-1} 
  & {}^\dagger{}E_2^{2,\nu-1} = E_2^{2,\nu-1} & \cdots
\\
  \vdots & \vdots & \vdots & \\
{}^\dagger{}E_2^{0,1} = E_2^{0,1} & {}^\dagger{}E_2^{1,1} = E_2^{1,1} 
  & {}^\dagger{}E_2^{2,1} = E_2^{2,1} & \cdots
\\
{}^\dagger{}E_2^{0,0} = E_2^{0,0} & {}^\dagger{}E_2^{1,0} = E_2^{1,0} 
  & {}^\dagger{}E_2^{2,0} = E_2^{2,0} & \cdots
\end{array}
$$
\label{f:dEvE}
\end{figure}
\end{small}
(The asterisk denotes no inclusion relation.)  From this we see that ${}^\dagger{}E_\infty^{p,q} = E_\infty^{p,q}$ for all $q < \nu$ and ${}^\dagger{}E_\infty^{0,\nu} \subset E_\infty^{0,\nu}$.  This yields Theorem \ref{T:hh}(c).

\begin{remark} \label{R:hh}
From \eqref{SE:dEvE}, we see that the image of the inclusion ${}^\dagger\tGr^0\bH^{\nu}(D,\sH^\ast_\ast) \inj \tGr^0{}H^{\nu}_\cI(D)$ in Theorem \ref{T:hh}(c) may be described as follows.  First note that the inclusion of ${}^\dagger{}E_1^{0,\nu} = H^{\nu}(D,\sH^0_0)$ into $E_1^{0,\nu} = H^{\nu}(\cC^{0,\sb},\bar\d)$ is given by Theorem \ref{T:E1}(b).  Second, 
$$
  \tGr^0{}H^{\nu}_\cI(D) \ = \ E^{0,\nu}_\infty \ = \ 
  \bigcap_{i=1}^{\nu} \tker\, \bd_i\,,
$$
where $\bd_1$ is defined on $E_1^{0,\nu}$, and each successive operator $\bd_{i+1}$ is defined on the kernel of the previous.  Third,
$$
  {}^\dagger\tGr^0\bH^{\nu}(D,\sH^\ast_\ast) \ = \ {}^\dagger{}E^{0,\nu}_\infty
  \ = \ H^{\nu}(D,\sH^0_0) \,\cap\, E^{0,\nu}_\infty
  \ = \ H^{\nu}(D,\sH^0_0) \,\cap\, \tGr^0{}H^{\nu}_\cI(D) \,.
$$
\end{remark}

\subsubsection*{Proof of Theorem \ref{T:hh}(a)}

Turning to the second spectral sequence ${}^\ddagger{}E$, the Poincar\'e lemma of Corollary \ref{C:PL2} implies $\cH^0 = \bC$ and $\cH^p = 0$ for all $0 < p < \nu$.  Therefore,
$$
  {}^\ddagger E^{p,q}_2 \ = \ 
  \left\{
  \begin{array}{ll}
    H^q(D,\bC) \,, \ & p=0\,,\\
    0 \,, \ & 0 < p < \nu \,,
  \end{array}
  \right.
$$
\cf Figure \ref{f:starE}.  
\begin{small}
\begin{figure}[!h] 
\caption[The ${}^\ddagger E_2$--page]{The ${}^\ddagger E_2$--page of the hypercohomology spectral sequence.}
$$ \renewcommand{\arraystretch}{1.5}
\begin{array}{ccccccc}
\vdots & \vdots & & \vdots & \vdots & \vdots & \\
H^2(D,\bC) & 0 & \cdots & 0 & H^2(D,\cH^{\nu}) & H^2(D, \cH^{\nu+1}) & \cdots \\
H^1(D,\bC) & 0 & \cdots & 0 & H^1(D,\cH^{\nu}) & H^1(D, \cH^{\nu+1}) & \cdots \\
H^0(D,\bC) & 0 & \cdots & 0 & H^0(D,\cH^{\nu}) & H^0(D, \cH^{\nu+1}) & \cdots \\
\end{array}
$$
\label{f:starE}
\end{figure}
\end{small}
(When considering Figure \ref{f:starE} it is important to recall that the differential $d^\ddagger_i$ `points' towards the northwest $\nwarrow$, while all other spectral sequence differentials considered in this paper `point' towards the southeast $\searrow$.)  Theorem \ref{T:hh}(a) follows.

\subsubsection*{Proof of Theorem \ref{T:hh}(b)}

Again consulting Figure \ref{f:starE} we see that the terms ${}^\ddagger{}E^{p,q}_\infty$ with $p+q = \nu$ are 
\begin{eqnarray*}
  {}^\ddagger{}E^{0,\nu}_\infty & = &  H^{\nu}(D,\bC) \,,\\
  {}^\ddagger{}E^{p,q}_\infty & = & 0 \quad\hbox{for }\ p,q>0 \,,\\ 
  {}^\ddagger{}E^{\nu,0}_\infty & = & 
  \tker\{ d^\ddagger_{\nu+1} : H^0(D,\cH^{\nu}) \to H^{\nu+1}(D,\bC) \} \,.
\end{eqnarray*}
Whence
$$
  {}^\ddagger\tGr^\sb \bH^{\nu}(D,\sH^\ast_\ast) \ = \ H^{\nu}(D,\bC) 
  \ \op \ \tker\{ d^\ddagger_{\nu+1} : H^0(D,\cH^{\nu}) \to H^{\nu+1}(D,\bC) \}
$$
and
$$
  H^{\nu}(D,\bC) \ \subset \ \bH^{\nu}(D,\sH^\ast_\ast) \,.
$$
Assertion (b) of Theorem \ref{T:hh} follows.

\appendix

\section{Examples} \label{S:egs}

We have seen that the eigenspace decomposition \eqref{E:Hev} plays a central role in the characteristic cohomology.  Here we present a number of examples illustrating the decomposition and the values $\mu$ and $\nu$ of \eqref{E:dfn_s} and \eqref{E:dfn_lo}, respectively.  The eigenspace decomposition is computed using Kostant's theorem on Lie algebra cohomology which is briefly reviewed in Appendix \ref{S:kostant}.  

This section contains several figures illustrating the decomposition, and I would like to make two comments on the interpretation of those figures.  First,  the decomposition \eqref{E:Hev} of $H^\ell(\fg_-,\bC)$ lies on the $\ell$--th diagonal.  Second, virtue of Lemma \ref{L:circE} and its analogs (such as Lemma \ref{L:starE}), these figures may be identified with those representing the spectral sequence pages ${}^\circ{}E_1$ (Figure \ref{f:circE1}), ${}^\star{}E_1$ and their analogs in Sections \ref{S:prfE2} and \ref{S:prfE1}.

\subsection{Adjoint varieties} \label{S:eg_adj}

Consider the case that $\cT_1 \subset \cT \check D$ is a contact distribution.  This is the case precisely when $G_\bC$ is simple and the minimal homogeneous embedding of $\check D$ is the $G_\bC$--orbit of the highest root line $\fg^{\tilde\a} \in \bP \fg_\bC$.  These are the \emph{adjoint varieties}, the compact, simply connected, homogeneous complex contact manifolds \cite{MR0137126}.  These examples are easily described by the geometry of the contact distribution; it is not necessary to appeal to representation theory.  In this case, the splitting \eqref{SE:split*} is 
$$
  \cT D^{\,*} \ = \ \tw^{1,0}_1 \,\op\, \tw^{1,0}_2 \,,
  \quad\hbox{with } \ 
  \tdim_\bC \tw^{1,0}_1 \,=\, 2c \ \hbox{ and } \ \tdim_\bC \tw^{1,0}_2 \,=\, 1 \,.
$$
Note that $\tw^{1,0}_2 = \tAnn(\cT_1)$.

Figures \ref{f:E0adj} and \ref{f:E1adj} depict the pages ${}^\star{}E_0$ and ${}^\star{}E_1$ of the spectral sequence introduced in Section \ref{S:filt} (and generalized in Section \ref{S:prfE1}).  
\begin{small}
\begin{figure}[h] 
\caption[The ${}^\star E_0^{\ell,-m}$--term for adjoint varieties]{The initial term  ${}^\star E_0^{\ell,-m} = \cA_\ell^{0,\ell-m}$ in the case that $\check D$ is an adjoint variety.}
$$
\renewcommand{\arraystretch}{1.3}
\begin{array}{cccccccc}
 \cA^{0,0} & \cA^{0,1}_1 & \cA^{0,2}_2 & \cA^{0,3}_3 & \cdots & \cA^{0,2c}_{2c} 
    & 0 & 0 \\
  & & \uparrow\hbox{\scriptsize{$\bar\del_0$}} 
    & \uparrow\hbox{\scriptsize{$\bar\del_0$}} &
    & \uparrow\hbox{\scriptsize{$\bar\del_0$}} &  \\ 
 0  & 0 & \cA^{0,1}_2 & \cA^{0,2}_3 & \cdots & \cA^{0,2c-1}_{2c} 
    & \cA^{0,2c}_{2c+1} & \cA^{0,2c+1}_{2c+2} \\
\end{array}
$$
\label{f:E0adj}
\end{figure}
\end{small}
When considering Figure \ref{f:E0adj}, recall that $\cA^{0,\ell}_\ell$ denotes the smooth sections of $\tw^{0,\ell}_\ell$, and $\cA^{0,\ell}_{\ell+1}$ denotes the smooth sections of $\tw^{0,\ell-1}_\ell \ot \tw^{0,1}_2$, \cf \eqref{E:twgrad}.  The nondegeneracy of the contact form implies that the algebraic differential $\bar\del_0 : \cA^{0,\ell}_{\ell+1} \to \cA^{0,\ell+1}_{\ell+1}$ is injective when $\ell \le c+1$ and surjective when $\ell \ge c+1$.  It follows from Lemma \ref{L:starE} that the ${}^\star{}E_1$--term of the spectral sequence is as depicted in Figure \ref{f:E1adj}.  Referring to the definitions \eqref{E:dfn_s} and \eqref{E:dfn_lo} we see that
\[
  \mu \,=\, \nu \,=\, c \,.
\]
\begin{small}
\begin{figure}[h] 
\caption[The term ${}^\star E_1^{\ell,-m}$ for adjoint varieties]{The term  ${}^\star E_1^{\ell,-m} = \cC^\infty(\overline{\sH^{\ell-m}_m})$ in the case that $\check D$ is an adjoint variety.}
$$
\renewcommand{\arraystretch}{1.5}
\begin{array}{ccccccc}
 \cC^\infty(\overline{\sH^0_0}) & \cdots 
     & \cC^\infty(\overline{\sH^{c}_{c}}) & 0  \\
 & & & 0 & \cC^\infty(\overline{\sH^{c+1}_{c+2}}) & \cdots  
    & \cC^\infty(\overline{\sH^{2c+1}_{2c+2}}) \\
\end{array}
$$
\label{f:E1adj}
\end{figure}
\end{small}

This is an example in which the resolution \eqref{E:BGGres} of Remark \ref{R:BGGres1} exists.  Indeed from Figure \ref{f:E1adj} we see that there exists a complex
\begin{equation} \nonumber 
  0 \ \to \ \cO \ \inj \ 
  \cC^\infty(\overline{\sH^0} ) \ \stackrel{\nabla^1}{\to} \cdots \stackrel{\nabla^1}{\to} 
  \cC^\infty( \overline{\sH^{c}} )\ \stackrel{\nabla^2}{\to} 
  \cC^\infty(\overline{\sH^{c+1}} ) \ \stackrel{\nabla^1}{\to} \cdots \stackrel{\nabla^1}{\to} 
  \cC^\infty(\overline{\sH^{2c+1}} ) \ \to \ 0 \,,
\end{equation}
where $\nabla^a$ denotes an operator of order $a$.  (This is the case $p=0$ in \eqref{E:BGGres}.)  To see that the complex is exact, if suffices to recall that the spectral sequence $\{ {}^\star E^{p,q}_i , \bar\del_i\}$ converges to the Dolbeault cohomology.  This resolution may be thought of as a Dolbeault analog of the Rumin complex \cite{BEGN, MR1046521}.  A similar argument gives \eqref{E:BGGres} for $p>0$.

\subsection{\bmath{Flag varieties $\check D = \tFlag(a,b,\bC^5)$}} \label{S:eg_A4}

If the compact dual is a Grassmannian, the IPR $\cT_1 = \cT{\check D}$ is trivial.  So we will consider a examples of the form $\check D = \tFlag(a,b,\bC^5)$.  (The case that $(a,b) = (1,4)$ is omitted as the compact dual $\check D$ is an adjoint variety; see Section \ref{S:eg_adj}.)  For these varieties
$$
  \ttE \ = \ \ttS^a \,+\, \ttS^b \,.
$$  
The nontrivial $\ttE$--eigenspaces $H^\ell_m$ for these two compact duals are computed by \eqref{E:Klo}; see Figures \ref{f:A4P12}--\ref{f:A4P23}.  The values of $\mu$ and $\nu$, determined by inspection of the figures, are listed in Table \ref{t:A4}.
\begin{table}[h] 
\caption{$(\nu,\mu)$ values for $\tFlag(a,b,\bC^5)$}\renewcommand{\arraystretch}{1.3}
\begin{tabular}{c|ccc}
  $\check D$ & $\tFlag(1,2,\bC^5)$ & $\tFlag(1,3,\bC^5)$ & $\tFlag(2,3,\bC^5)$ \\
  \hline
  $(\nu,\mu)$ & $(1,3)$ & $(2,4)$ & $(1,2)$
\end{tabular}
\label{t:A4}
\end{table}
\begin{small}
\begin{figure}[h] 
\caption{Nontrivial $H^\ell_m$ for $\check D = \tFlag(1,2,\bC^5)$}
$$
\renewcommand{\arraystretch}{1.3}
\begin{array}{ccccccccccc}
  H^0_0 & H^1_1 & H^2_2 & H^3_3 & 0 & 0 & 0 & 0 & 0 & 0 & 0  \\
  0 & 0 & 0 & H^2_3 & H^3_4 & H^4_5 & 0 & 0 & 0 & 0 & 0  \\
  0 & 0 & 0 & 0 & 0 & H^3_5 & H^4_6 & H^5_7 & 0 & 0 & 0  \\
  0 & 0 & 0 & 0 & 0 & 0 & 0 & H^4_7 & H^5_8 & H^6_9 & H^7_{10}  \\
\end{array}
$$
\label{f:A4P12}
\end{figure}
\end{small}
\begin{small}
\begin{figure}[h] 
\caption{Nontrivial $H^\ell_m$ for $\check D = \tFlag(1,3,\bC^5)$}
$$
\renewcommand{\arraystretch}{1.3}
\begin{array}{ccccccccccc}
  H^0_0 & H^1_1 & H^2_2 & H^3_3 & H^4_4 & 0 & 0 & 0 & 0 & 0 & 0 \\
  0 & 0 & 0 & 0 & H^3_4 & H^4_5 & H^5_6 & 0 & 0 & 0 & 0 \\
  0 & 0 & 0 & 0 & 0 & 0 & H^4_6 & H^5_7 & H^6_8 & H^7_9 & H^8_{10} \\
\end{array}
$$
\label{f:A4P13}
\end{figure}
\end{small}
\begin{small}
\begin{figure}[h] 
\caption{Nontrivial $H^\ell_m$ for $\check D = \tFlag(2,3,\bC^5)$}
$$
\renewcommand{\arraystretch}{1.3}
\begin{array}{ccccccccccccc}
  H^0_0 & H^1_1 & H^2_2 & 0 & 0 & 0 & 0 & 0 & 0 & 0 & 0 & 0 & 0 \\
  0 & 0 & 0 & H^2_3 & H^3_4 & 0 & 0 & 0 & 0 & 0 & 0 & 0 & 0 \\
  0 & 0 & 0 & 0 & 0 & H^3_5 & H^4_6 & H^5_7 & 0 & 0 & 0 & 0 & 0 \\
  0 & 0 & 0 & 0 & 0 & 0 & 0 & 0 & H^5_8 & H^6_9 & 0 & 0 & 0 \\
  0 & 0 & 0 & 0 & 0 & 0 & 0 & 0 & 0 & 0 & H^6_{10} & H^7_{11} & H^8_{12} \\
\end{array}
$$
\label{f:A4P23}
\end{figure}
\end{small}
%

\subsection{\bmath{The exceptional group $G_2$}} \label{S:eg_G2}

The compact dual $\check D = G_2(\bC)/P_2$ is an adjoint variety (Section \ref{S:eg_adj}), so here we will consider only the compact duals $\cQ^5 = G_2/P_1$, which has grading element $\ttE = \ttS^1$; and $G_2/P_{1,2} = G_2/B$, which has grading element $\ttE = \ttS^1 +\ttS^2$.  The nontrivial $\ttE$--eigenspaces $H^\ell_m$ for these two compact duals are computed by \eqref{E:Klo}, and are depicted in Figures \ref{f:G2P1} and \ref{f:G2P12}.  From these figures we see that
$$
  \nu \ = \ 1
$$
in both examples.

\begin{small}
\begin{figure}[h] 
\caption{Nontrivial $H^\ell_m$ for $\check D = G_2/P_1$}
$$
\renewcommand{\arraystretch}{1.3}
\begin{array}{ccccccccccc}
  H^0_0 & H^1_1 & 0 & 0 & 0 & 0 & 0 & 0 & 0 & 0 & 0  \\
  0 & 0 & 0 & 0 & 0 & 0 & 0 & 0 & 0 & 0 & 0 \\
  0 & 0 & 0 & 0 & H^2_4 & 0 & 0 & 0 & 0 & 0 & 0 \\
  0 & 0 & 0 & 0 & 0 & 0 & H^3_6 & 0 & 0 & 0 & 0 \\
  0 & 0 & 0 & 0 & 0 & 0 & 0 & 0 & 0 & 0 & 0 \\
  0 & 0 & 0 & 0 & 0 & 0 & 0 & 0 & 0 & H^4_9 & H^5_{10}
\end{array}
$$
\label{f:G2P1}
\end{figure}
\end{small}

Consider the case that $\check D = G_2/P_1$.  From Figure \ref{f:G2P1} we see that the resolution \eqref{E:BGGres} exists.  In the case that $p=0$ the resolution is of the form
\begin{eqnarray*}
  0 \ \to \ \cO \ \inj \ 
  \cC^\infty( \overline{\sH^0} ) & \stackrel{\nabla^1}{\to} & 
  \cC^\infty( \overline{\sH^1} ) \ \stackrel{\nabla^3}{\to} \
  \cC^\infty( \overline{\sH^2} ) \ \stackrel{\nabla^2}{\to} \ 
  \cC^\infty( \overline{\sH^3} ) \\
  & \stackrel{\nabla^3}{\to} &
  \cC^\infty( \overline{\sH^4} ) \ \stackrel{\nabla^1}{\to} \ 
  \cC^\infty( \overline{\sH^5} ) \ \to \ 0 
\end{eqnarray*}
with $\nabla^{a}$ a $G_\bR$--invariant differential operator of order $a$.  (See \cite[Section 5]{BEGN} for a discussion of this resolution in a related setting.)
\begin{small}
\begin{figure}[h] 
\caption{Nontrivial $H^\ell_m$ for $\check D = G_2/P_{1,2}$}
$$
\renewcommand{\arraystretch}{1.3}
\begin{array}{ccccccccccccccccc}
  H^0_0 & H^1_1 & 0 & 0 & 0 & 0 & 0 & 0 & 0 & 0 & 0 & 0 & 0 & 0 & 0 & 0 & 0 \\
  0 & 0 & 0 & H^2_3 & 0 & 0 & 0 & 0 & 0 & 0 & 0 & 0 & 0 & 0 & 0 & 0 & 0 \\
  0 & 0 & 0 & 0 & 0 & 0 & 0 & 0 & 0 & 0 & 0 & 0 & 0 & 0 & 0 & 0 & 0 \\
  0 & 0 & 0 & 0 & 0 & H^2_5 & 0 & 0 & 0 & 0 & 0 & 0 & 0 & 0 & 0 & 0 & 0 \\
  0 & 0 & 0 & 0 & 0 & 0 & 0 & 0 & 0 & 0 & 0 & 0 & 0 & 0 & 0 & 0 & 0 \\
  0 & 0 & 0 & 0 & 0 & 0 & 0 & 0 & H^3_8 & 0 & 0 & 0 & 0 & 0 & 0 & 0 & 0 \\
  0 & 0 & 0 & 0 & 0 & 0 & 0 & 0 & 0 & 0 & 0 & 0 & 0 & 0 & 0 & 0 & 0 \\
  0 & 0 & 0 & 0 & 0 & 0 & 0 & 0 & 0 & 0 & 0 & H^4_{11} & 0 & 0 & 0 & 0 & 0 \\
  0 & 0 & 0 & 0 & 0 & 0 & 0 & 0 & 0 & 0 & 0 & 0 & 0 & 0 & 0 & 0 & 0 \\
  0 & 0 & 0 & 0 & 0 & 0 & 0 & 0 & 0 & 0 & 0 & 0 & 0 & H^4_{13} & 0 & 0 & 0 \\
  0 & 0 & 0 & 0 & 0 & 0 & 0 & 0 & 0 & 0 & 0 & 0 & 0 & 0 & 0 & H^5_{15} & H^6_{16} \\
\end{array}
$$
\label{f:G2P12}
\end{figure}
\end{small}

Consider the case that $\check D = G_2/B$.  From Figure \ref{f:G2P12} we see that this is also an example in which the resolution \eqref{E:BGGres} exists.  In the case that $p=0$ the resolution is of the form
\begin{eqnarray*}
  0 \ \to \ \cO \ \inj \ 
  \cC^\infty( \overline{\sH^0} ) & \stackrel{\nabla^1}{\to} & 
  \cC^\infty( \overline{\sH^1} ) \ \stackrel{\nabla^2}{\to} \
  \cC^\infty( \overline{\sH^2} ) \ \stackrel{\nabla^3}{\to} \ 
  \cC^\infty( \overline{\sH^3} ) \\
  & \stackrel{\nabla^3}{\to} &
  \cC^\infty( \overline{\sH^4} ) \ \stackrel{\nabla^2}{\to} \ 
  \cC^\infty( \overline{\sH^5} ) \ \stackrel{\nabla^1}{\to} \ 
  \cC^\infty( \overline{\sH^6} ) \ \to \ 0 \,.
\end{eqnarray*}

\section{Kostant's Theorem} \label{S:kostant}

This section is a terse summary of Kostant's theorem on Lie algebra cohomology \cite[Theorem 5.14]{MR0142696}.  We restrict the discussion to cohomology with coefficients in the trivial representation $\bC$. (Kostant's theorem addresses the more general setting of coefficients in an arbitrary irreducible $\fg_\bC$--representation.)  The theorem describes the $\fg_0$--module structure of $H^\sb(\fg_-,\bC)$ as follows.

Let $\{ \w_1,\ldots,\w_r\} \subset \fh^*$ denote the \emph{fundamental weights} of $(\fg_\bC,\Sigma)$.  Let $\wtL = \wtL(\fg_\bC) = \tspan_\bZ\{\w_1,\ldots,\w_r\}$ denote the \emph{weight lattice}.  Then a weight $\lambda = n^i\w_i \in \wtL$ is \emph{$\fg_\bC$--dominant} if $n^i \ge0$ for all $i$.  Similarly, a weight is \emph{$\fg_0$--dominant} if $n^i \ge0$ for all $i \not \in I$, \cf \eqref{E:Sigmav}.  Let $\wtD(\fg_\bC) \subset \wtD(\fg_0)$ denote the respective sets of dominant weights.

The Weyl group has the property that $W(\wtL) = \wtL$.   The set $W^\fp$ indexing Schubert varieties (Section \ref{S:schub}) may be characterized by
$$
  W^\fp \ = \ \{ w \in W \ | \ w(\wtD(\fg_\bC)) \,\subset\,\wtD(\fg_0) \} \,,
$$
\cf\cite[\S5.13]{\Kostant1}.  One of the simplest ways to determine the elements of $W^\fp$ is to use the fact that they are in bijective correspondence with the orbit of 
$$
  \rho_0 \ \dfn \ \sum_{i \in I} \w_i
$$
under the Weyl group $W$, via the assignment $w \mapsto w^{-1} \rho_0$.    Let
$$
  W^\fp(\ell) \ = \ \{ w \in W^\fp \ | \ |w| = \ell\}
$$
denote the elements of length $\ell$.

Let 
$$
  \rho \ = \ \sum_i \w_i \ \in \ \wtL \,.
$$
Given $w \in W$ define 
\begin{equation} \label{E:rhow}
  \rho_w \ = \ \rho \,-\, w(\rho) \ \in \ \wtL \,.
\end{equation}
Then $-\rho_w \in \wtD(\fg_0)$; let $H_{w}$ denote the irreducible $\fg_0$--module of \emph{lowest} weight $\rho_w$.  (Equivalently, the dual $H_{w}^*$ is the irreducible $\fg_0$--module of highest weight $-\rho_w$.)  By Kostant's \cite[Theorem 5.14]{MR0142696}, the Lie algebra cohomology
\begin{equation} \label{E:kostant}
  H^\ell( \fg_- , \bC) \ = \ \bigoplus_{w \in W^\fp(\ell)} H_w
\end{equation}
as a $\fg_0$--module.  Moreover, $\rho_w = \rho_v$ if and only if $w = v$; that is, the multiplicity of $H_w$ in $H^\sb(\fg_-,\bC)$ is one.  
Kostant's \eqref{E:kostant} determines the $\ttE$--eigenspace decomposition \eqref{E:Hev} and the integer $\nu$ of \eqref{E:dfn_lo} as follows.  Precisely,
\begin{equation} \label{E:Hw_ev}
  H^\ell_m \ = \ \bigoplus_{\mystack{w \in W^\fp(\ell)}{\rho_w(\ttE) = m}} H_w \,.
\end{equation}
Thus,
\begin{equation} \label{E:Klo}
  \nu \ = \ 
  \tmax\{ \ell \ | \ \rho_w(\ttE) = \ell \,,\ \forall \ w \in W^\fp(\ell) \} \,.
\end{equation}

\def\cprime{$'$} \def\Dbar{\leavevmode\lower.6ex\hbox to 0pt{\hskip-.23ex
  \accent"16\hss}D}

\end{document}

%% file: macros.tex
\def\Cap{\v{C}ap}

\def\Soucek{Sou\v{c}ek}
\def\Vladimir{Vladim\'{i}r}
\def\cf{cf.~}

\newcommand{\hsp}[1]{{\hbox{\hspace{#1}}}}
\newcommand{\bmath}[1]{{\bf \boldmath {#1} \unboldmath}}

\newcommand{\mystack}[2]{\ensuremath{ \substack{ \hbox{\tiny{${#1}$}} \\ \hbox{\tiny{${#2}$}} }} }

\def\a{\alpha}  

\def\d{\delta}  
\def\e{\varepsilon}

\def\m{\mu}
\def\n{\nu}
\def\s{\sigma}

\def\w{\omega}

\def\cA{\mathcal{A}} 
 
\def\fa{\mathfrak{a}}

 \def\tAnn{\mathrm{Ann}}
\def\tAut{\mathrm{Aut}}

\def\fb{\mathfrak{b}} 
 
\def\bC{\mathbb C} \def\cC{\mathcal C}

\def\bc{\mathbf{c}}

 \def\bd{\mathbf{d}}
\def\td{\mathrm{d}}

 \def\tdim{\mathrm{dim}}
 
 \def\ttE{\mathtt{E}}

\def\tFlag{\mathrm{Flag}}

\def\tGr{\mathrm{Gr}}
\def\fg{{\mathfrak{g}}}

\def\bH{\mathbb{H}}
\def\cH{\mathcal H} 
 \def\sH{\mathscr{H}}

\def\fh{\mathfrak{h}}

\def\ffi{\mathfrak{i}}
  \def\cI{\mathcal I} \def\sI{\mathscr{I}}

 \def\tim{\mathrm{im}}
 
  \def\fj{\mathfrak{j}}

 \def\tker{\mathrm{ker}}

\def\tmax{\mathrm{max}} \def\tmin{\mathrm{min}}
\def\tmod{\mathrm{mod}}

\def\fn{\mathfrak{n}}

 \def\cO{\mathcal O}

\def\bP{\mathbb P}

\def\fp{\mathfrak{p}}

 \def\cQ{\mathcal Q}

\def\bR{\mathbb R}

\def\trank{\mathrm{rank}}

\def\ttS{\mathtt{S}} 
\def\fs{\mathfrak{s}}
 
\def\tss{\mathrm{ss}}
\def\tSL{\mathrm{SL}}

 \def\tspan{\mathrm{span}}

\def\cT{\mathcal T}  
 
\def\ft{\mathfrak{t}}

\def\sU{\mathscr{U}}

\def\fv{\mathfrak{v}}

\def\bx{\mathbf{x}}

\def\by{\mathbf{y}}
   \def\bZ{\mathbb Z}

\def\half{\tfrac{1}{2}}

\def\tand{\quad\hbox{and}\quad}
\def\del{\partial}
\def\dfn{\stackrel{\hbox{\tiny{dfn}}}{=}}
\def\sb{{\hbox{\tiny{$\bullet$}}}}

\def\inj{\hookrightarrow}

\def\op{\oplus}
\def\ot{\otimes}

\def\tw{\hbox{\small $\bigwedge$}}

\def\wtL{{\Lambda_\mathrm{wt}}}
\def\wtD{{\Lambda_\mathrm{wt}^+}}

\newcounter{numcnt}

\newcounter{cnt}

\newcounter{acnt}

\newcounter{Acnt}

\newcounter{icnt}

\newcounter{Icnt}

\newcounter{exam_cnt}

\newcounter{mccnt}


%% file: thms.tex
\newtheorem{corollary}[equation]{Corollary}
\newtheorem*{corollary*}{Corollary}
\newtheorem{lemma}[equation]{Lemma}
\newtheorem*{lemma*}{Lemma}

\newtheorem*{proposition*}{Proposition}
\newtheorem{theorem}[equation]{Theorem}
\newtheorem*{theorem*}{Theorem}

\theoremstyle{definition}

\newtheorem*{boldQ*}{Question}
\newtheorem*{boldP*}{Problem}

\theoremstyle{remark}
\newtheorem*{assume*}{Assume}
\newtheorem*{answer*}{Answer}

\newtheorem*{claim*}{Claim}

\newtheorem*{definition*}{Definition}

\newtheorem*{example*}{Example}
\newtheorem*{hint*}{Hint}
\newtheorem*{notation*}{Notation}
\newtheorem{remark}[equation]{Remark}
\newtheorem*{remark*}{Remark}
\newtheorem*{remarks*}{Remarks}
\newtheorem*{fact*}{Fact}
\newtheorem*{emphL*}{Lemma}

\newtheorem*{emphQ*}{Question}
\newtheorem*{emphA*}{Answer}

\numberwithin{HWeq}{section}
\theoremstyle{definition}

